\documentclass[12pt]{amsart}
\usepackage{amssymb}
\usepackage{eucal}
\usepackage[bookmarks=false]{hyperref}
\usepackage{bm}
\usepackage[small,nohug,heads=vee]{diagrams}
\diagramstyle[labelstyle=\scriptstyle]

\allowdisplaybreaks

\newtheorem{theorem}{Theorem}[section]
\newtheorem{lemma}[theorem]{Lemma}
\newtheorem{proposition}[theorem]{Proposition}
\newtheorem{corollary}[theorem]{Corollary}

\newtheorem{tablesymb}[theorem]{Table of Symbols}
\newtheorem*{theorem:ExpCoef}{Theorem~\ref{T:ExpCoef}}
\newtheorem*{corollary:ExpCor}{Corollary~\ref{C:Expcorallary}}
\newtheorem*{theorem:LogCoef}{Theorem~\ref{T:LogCoef}}
\newtheorem*{corollary:LogCor}{Corollary~\ref{C:LogCoefbottomrow}}
\newtheorem*{theorem:ZetaValues}{Theorem~\ref{T:ZetaValues}}
\newtheorem*{corollary:ZetaCor}{Corollary~\ref{C:ZetaValues}}
\newtheorem*{theorem:Transcendence}{Theorem~\ref{T:Transcendence}}
\newtheorem*{corollary:TranscedenceCor}{Corollary~\ref{C:Transcendence}}

\theoremstyle{definition}

\newtheorem{example}[theorem]{Example}

\theoremstyle{remark}
\newtheorem{remark}[theorem]{Remark}

\newcommand{\G}{\ensuremath \mathbb{G}}

\newcommand{\C}{\ensuremath \mathbb{C}}
\newcommand{\Z}{\ensuremath \mathbb{Z}}

\newcommand{\BB}{\mathbb{B}}

\newcommand{\F}{\ensuremath \mathbb{F}}

\newcommand{\TT}{\mathbb{T}}

\newcommand{\bA}{\mathbf{A}}
\newcommand{\bH}{\mathbf{H}}
\newcommand{\bK}{\mathbf{K}}
\newcommand{\ba}{\mathbf{a}}
\newcommand{\bb}{\mathbf{b}}

\newcommand{\bd}{\mathbf{d}}
\newcommand{\bu}{\mathbf{u}}

\newcommand{\bg}{\mathbf{g}}
\newcommand{\bh}{\mathbf{h}}

\newcommand{\bz}{\mathbf{z}}

\newcommand{\bw}{\mathbf{w}}
\newcommand{\bk}{\mathbf{k}}
\newcommand{\bp}{\mathbf{p}}
\newcommand{\balpha}{\bm{\alpha}}
\newcommand{\bgamma}{\bm{\gamma}}

\newcommand{\cE}{\mathcal{E}}
\newcommand{\cF}{\mathcal{F}}
\newcommand{\cG}{\mathcal{G}}

\newcommand{\cO}{\mathcal{O}}

\newcommand{\cR}{\mathcal{R}}

\newcommand{\cU}{\mathcal{U}}
\newcommand{\cZ}{\mathcal{Z}}
\newcommand{\cM}{\mathcal{M}}

\newcommand{\tcG}{\widetilde{\mathcal{G}}}
\newcommand{\tcZ}{\widetilde{\mathcal{Z}}}

\newcommand{\fa}{\mathfrak{a}}

\newcommand{\fp}{\mathfrak{p}}

\newcommand{\isom}{\ensuremath \cong}
\newcommand{\inv}{\ensuremath ^{-1}}

\newcommand{\on}{\ensuremath ^{\otimes n}}

\DeclareMathOperator{\Span}{Span}

\DeclareMathOperator{\Mat}{Mat}
\DeclareMathOperator{\divisor}{div}
\DeclareMathOperator{\End}{End}
\DeclareMathOperator{\Exp}{Exp}
\DeclareMathOperator{\Log}{Log}

\DeclareMathOperator{\Fr}{Fr}
\DeclareMathOperator{\Gal}{Gal}
\DeclareMathOperator{\Hom}{Hom}

\DeclareMathOperator{\ord}{ord}

\DeclareMathOperator{\Res}{Res}
\DeclareMathOperator{\RES}{RES}
\DeclareMathOperator{\sgn}{sgn}
\DeclareMathOperator{\Spec}{Spec}

\newcommand{\twist}{^{(1)}}

\newcommand{\twistinv}{^{(-1)}}
\newcommand{\twisti}{^{(i)}}
\newcommand{\twistk}[1]{^{(#1)}}

\newcommand{\oa}{\overline{a}}

\newcommand{\ob}{\overline{b}}

\newcommand{\ophi}{\overline{\phi}}

\newcommand{\power}[2]{{#1 [[ #2 ]]}}

\newcommand{\pd}{\partial}

\newcommand{\tlambda}{\widetilde{\lambda}}

\newcommand{\tpi}{\widetilde{\pi}}

\newcommand{\tsgn}{\widetilde{\sgn}}

\def\XXint#1#2#3{{\setbox0=\hbox{$#1{#2#3}{\int}$ }
\vcenter{\hbox{$#2#3$ }}\kern-.6\wd0}}

\begin{document}

\title[Drinfeld Module Zeta Values]{Special Zeta Values using Tensor Powers of Drinfeld Modules}
\author{Nathan Green}
\address{Department of Mathematics, Texas A{\&}M University, College Station,
TX 77843, USA}
\email{jaicouru@gmail.com}

\subjclass[2010]{}

\date{}

\thanks{This project was partially supported by NSF Grant DMS-1501362.}

\begin{abstract}
We study tensor powers of rank 1 sign-normalized Drinfeld $\mathbf{A}$-modules, where $\mathbf{A}$ is the coordinate ring of an elliptic curve over a finite field.  Using the theory of vector valued Anderson generating functions, we give formulas for the coefficients of the logarithm and exponential functions associated to these $\mathbf{A}$-modules.  We then show that there exists a vector whose bottom coordinate contains a Goss zeta value, whose evaluation under the exponential function is defined over the Hilbert class field.  This allows us to prove the transcendence of Goss zeta values and periods of Drinfeld modules as well as the transcendence of certain ratios of those quantities.
\end{abstract}

\date{\today}
\maketitle

\tableofcontents

\section{Introduction} \label{S:Intro}
The Carlitz module has been studied extensively since Carlitz introduced it (see \cite{Carlitz35}) and we have explicit formulas for many objects related to its arithmetic.  In particular, we have formulas for the coefficients of the exponential and logarithm functions and many formulas for special values of its zeta function and $L$-functions.  Many years after Carlitz's original work, Drinfeld introduced the notion of Drinfeld modules which serve as important generalizations of the Carlitz module (see \cite{Goss} for a thorough account).  Since their introduction, much work has been done to develop an explicit theory for Drinfeld modules which parallels that for the Carlitz module, notably work by Anderson in \cite{And94} and \cite{And96}, Thakur in \cite{Thakur92} and \cite{Thakur93}, Dummit and Hayes in \cite{DummitHayes94}, and Hayes in \cite{Hayes79}.

Further generalizing the Carlitz module, Anderson introduced $n$th tensor powers of Drinfeld modules in \cite{And86}, which serve as $n$-dimensional analogues of (1-dimensional) Drinfeld modules (see \cite{BP16} for a thorough introduction to tensor powers of Drinfeld modules).  In their remarkable paper \cite{AndThak90}, Anderson and Thakur develop much of the explicit theory for the arithmetic of the $n$th tensor power of the Carlitz module, including recursive formulas for the coefficients of the exponential and logarithm functions.  Notably, their techniques allow them to connect special evaluations of the logarithm function to function field zeta values, which are defined in the case of the Carlitz module to be
\[\zeta(s) = \sum_{\substack{a\in \F_q[\theta] \\ a\text{ monic}}} \frac{1}{a^s}, \quad s\in \Z,\]
where $q$ is a prime power.  They find (\cite[Thm. 3.8.3]{AndThak90}) for $1\leq s\leq q -1$ (as a special case) that
\[\Log\on_C \left (\begin{matrix}
0  \\
\vdots \\
0  \\
1
\end{matrix}\right )
=
\left (\begin{matrix}
*  \\
\vdots \\
*  \\
\zeta(s)
\end{matrix}\right ),\]
where $\Log\on_C$ is the logarithm function associated to the $n$th tensor power of the Carlitz module.  In recent years, there has been a surge of work using Drinfeld modules to study zeta functions, $L$-functions, and their special values over functions fields (see \cite{AnglesNgoDacRibeiro16a}-\cite{AnglesSimon}, \cite{LutesThesis}, \cite{LP13}, \cite{Pellarin08}-\cite{Perkins14b}, \cite{Taelman12}).

To state the results of the present paper, we give a short review of rank 1 sign-normalized Drinfeld modules over affine coordinate rings of elliptic curves $E/\F_q$.  Let $\bA = \F_q[t,y]$, where $t$ and $y$ satisfy a cubic Weierstrass equation for $E$.  Define an isomorphic copy of $\bA$ with variables $\theta$ and $\eta$ satisfying the same cubic Weierstrass equation for $E$, which we denote as $A = \F_q[\theta,\eta]$.  Let $K$ be the fraction field of $A$, let $K_\infty$ be the completion of $K$ at the infinite place, an let $\C_\infty$ be the completion of an algebraic closure of $K_\infty$.  Denote $H$ as the Hilbert class field of $K$, which we can take to be a subfield of $K_\infty$.  For an algebraically closed field $L$, if we let $L[\tau]$ denote the ring of twisted polynomials in the $q$th power Frobenius endomorphism~$\tau$, then a rank 1 sign-normalized Drinfeld module is an $\bA$-module homomorphism
\[\rho:\bA\to L[\tau]\]
which satisfies certain naturally defined conditions (see \S \ref{S:Review}).   There is a point $V := (\alpha,\beta)\in E(H)$ associated to $\rho$ called the Drinfeld divisor, satisfying the equation on $E$
\begin{equation}\label{Veq1}
V\twist - V + \Xi = \infty,
\end{equation}
where $\Xi = (\theta,\eta)\in K(E)$ and $V^{(1)} = (\alpha^q,\beta^q)$.  Using the norm defined in \S \ref{S:Background}, we require that $\lvert \alpha\rvert,\lvert\beta\rvert >0$ so that $V$ is uniquely determined.  By \eqref{Veq1}, there exists a function $f\in H(t,y)$ with divisor
\[\divisor(f) = (V\twist) - (V) + (\Xi) - (\infty),\]
with suitable normalization, called the shtuka function.

In this paper we continue the study of tensor powers of rank 1 sign-normalized Drinfeld $\bA$-modules which was commenced by the author in \cite{GreenPeriods17}.  These tensor powers provide a further generalization of the Carlitz module and are examples of Anderson $\bA$-modules.  If we define $\Mat_n(L)[\tau]$ to be the ring of twisted polynomials in the $q$th power Frobenius endomorphism~$\tau$, which we extend to matrices entry-wise, then an $n$-dimensional Anderson $\bA$-module is an $\bA$-module homomorphism
\[\rho:\bA \to \Mat_n(L)[\tau]\]
satisfying certain naturally defined conditions (see \S \ref{S:Review}).  We will denote the $n$th tensor power of the Drinfeld module $\rho$ as $\rho\on$ and denote the exponential and logarithm functions connected with it as $\Exp_\rho\on$ and $\Log_\rho\on$ respectively, noting that both functions can be represented as power series in $\power{\Mat_n(H)}{\bz}$ for $\bz \in \C_\infty^n$.  The construction and basic properties of tensor powers of rank 1 Drinfeld modules are studied by the author in \cite{GreenPeriods17} and we will refer frequently to results from it for our present considerations.  In an effort to make the present paper as self-contained as possible, we recall the necessary facts from \cite{GreenPeriods17} in \S \ref{S:Review}.

The main theorems of the present paper give explicit formulas for the coefficients of the exponential and the logarithm function associated to tensor powers of rank 1 sign-normalized Drinfeld modules, show that evaluating the exponential function at a special vector with a zeta value in its bottom coordinate gives a vector in $H^n$.  As an application of the main theorems we use techniques of Yu from \cite{Yu97} to show that these zeta values and the periods connected to the Drinfeld module are transcendental.  This generalizes both the work of Thakur on Drinfeld modules and zeta values in \cite{Thakur93} as well as that of Anderson and Thakur on tensor powers of the Carlitz module in \cite{AndThak90}.  The methods which Anderson and Thakur apply to obtain formulas for the coefficients for the exponential and logarithm functions for tensor powers of the Carlitz module involve recursive matrix calculations, which allow them to analyze a particular coordinate of those coefficients.  In the case of tensor powers of Drinfeld modules, however, the matrices involved are much more complicated and do not give clean formulas as they do in the Carlitz case.  We develop new techniques to analyze the coefficients of the logarithm and exponential function inspired partially by work of Papanikolas and the author in \cite{GP16} and partially by ideas of Sinha in \cite{Sinha97}.  Further, Anderson and Thakur use special polynomials (called Anderson-Thakur polynomials) in \cite{AndThak90} to relate evaluations of the logarithm function to zeta values.  It is not yet clear how to generalize these Anderson-Thakur polynomials to tensor powers of Drinfeld modules, and so instead we use a generalization of techniques developed by Papanikolas and the author in \cite{GP16} to prove formulas for zeta values.  We comment that this technique allows us to study zeta values only for $1\leq s\leq q-1$; developing techniques to study zeta values for all $n\geq 1$ is a topic of ongoing study (see Remark \ref{R:largezetavalues}).

We begin by setting out the notation and background in \S \ref{S:Background}, then in \S \ref{S:Review} we give a brief review of the theory of tensor powers of rank 1 Drinfeld $\bA$-modules and vector-valued Anderson generating functions as laid out in \cite{GreenPeriods17}.  In particular, for a fixed dimension $n\geq 1$ we recall the functions $g_i,h_i\in H(t,y)$ for $1\leq i\leq n$ which form convenient bases for the $\bA$-motives $M$ and $N$ defined in \cite[\S3]{GreenPeriods17} and recount some of their properties (see Proposition \ref{P:review}).  In section \S \ref{S:ExpCoef} we move on to analyzing the coefficients of the exponential function $\Exp_\rho\on$ associated to tensor powers of rank 1 Drinfeld $\bA$-modules.  Define the Frobenius twisting automorphism for $g = \sum c_{j,k} t^jy^k\in  L[t,y]$ to be
\begin{equation}\label{twistdef}
g\twist = \sum c_{j,k}^q t^jy^k,
\end{equation}
and let $g\twisti$ denote the $i$th iteration of twisting.  First, we define functions for $1\leq \ell \leq n$ and $i\geq 1$
\[\gamma_{i,\ell} = g_\ell/(ff\twist \dots f\twistk{i-1})^n\]
and find that there is a unique expression for $\gamma_{i,\ell}$ of the form
\[\gamma_{i,\ell} = c_{\ell,1}g_1\twisti +c_{\ell,2}g_2\twisti + \dots c_{\ell,n}g_n\twisti +  \sum_{j,k} d_{j,k} \alpha_{j,k},\]
for $c_{\ell,m},d_{j,k}\in H$, where the functions $\alpha_{j,k}\in (t,y)$ satisfy naturally defined conditions given in \S \ref{S:ExpCoef}.  We denote $C_i = \langle c_{j,k}\rangle$, and we obtain our first main theorem about the coefficients of the exponential function.

\begin{theorem:ExpCoef}\label{T:ExpCoef}
For dimension $n\geq 2$ and $\bz\in \C_\infty$, if we write
\[\Exp_\rho\on(\bz) = \sum_{i=0}^\infty Q_i \bz\twisti,\]
then for $i\geq 0$, the exponential coefficients $Q_i = C_i$ and $Q_i\in \Mat_n(H)$.
\end{theorem:ExpCoef}
We prove this theorem by observing a recursive matrix equation which uniquely identifies the coefficients of the exponential function (see Lemma \ref{L:Recurrencelemma}), and then proving that the matrices $C_i$ satisfy the recursive equation.  After a bit more analysis, we obtain more exact formulas for the first column of $Q_i$.

\begin{corollary:ExpCor}\label{C:Expcorallary}
For $z\in \C_\infty$ we have the expression
\[\Exp_\rho\on
\left (\begin{matrix}
z\\
0 \\
\vdots\\
 0
\end{matrix}\right )=
\left (\begin{matrix}
z\\
0 \\
\vdots\\
 0
\end{matrix}\right )+
\sum_{i=0}^\infty 
\frac{z^{q^i}}{g_1\twisti(ff\twist\dots f\twistk{i-1})^n}\cdot 
\left (\begin{matrix}
g_1\\
g_2 \\
\vdots\\
g_n
\end{matrix}\right )\Bigg|_{\Xi\twisti}.\]
\end{corollary:ExpCor}

Next, we transition to studying the coefficients of the logarithm function in \S \ref{S:LogCoef}.  Our main technique in this section involves proving the commutativity of diagram \eqref{maindiagram}, which is inspired by work of Sinha in \cite{Sinha97}.  We then define a single variable function which, using the machinery from the diagram, allows us to recover the logarithm function.  This gives formulas for the logarithm coefficients in terms of residues of quotients of the functions $g_i$, $h_i$ and $f$.

\begin{theorem:LogCoef}\label{T:LogCoef}
For $\bz$ inside the radius of convegence of $\Log_\rho\on$, if we let
\[
\Log_\rho\on(\bz) = \sum_{i=0}^\infty P_i \bz\twisti
\]
for $n\geq 2$ and let $\lambda$ be the invariant differential on $E$, then $P_i \in \Mat_n(H)$ for $i\geq 0$ and
\[
P_i = \left \langle \Res_\Xi\left (\frac{g_j h_{n-k+1}\twisti}{(ff\twist \dots f\twisti)^n}\lambda\right )\right  \rangle_{1\leq j,k\leq n}.
\]
\end{theorem:LogCoef}

With a little further analysis we obtain cleaner formulas formulas for the bottom row of the logarithm coefficients.

\begin{corollary:LogCor}\label{C:LogCoefbottomrow}
For the coefficients $P_i$ of the function $\Log_\rho\on$, the bottom row of $P_i$, for $i\geq 0$, can be written as
\[
\left \langle \frac{h_{n-k+1}\twisti}{h_1(f\twist \dots f\twisti)^n}\bigg|_\Xi \right \rangle_{1\leq k\leq n}.
\]
\end{corollary:LogCor}

In section \S \ref{S:ZetaValues} we show that evaluating the exponential function at a special vector with a Goss zeta value in its bottom coordinate is in $H^n$.  To state our results, we recall the extension of a rank 1 sign-normalized Drinfeld module $\rho$ to integral ideals $\fa\subset A$ due to Hayes~\cite{Hayes79} (see \S\ref{S:ZetaValues}), which maps $\fa \mapsto \rho_\fa \in H[\tau]$.  We define $\partial(\rho_\fa)$ to be the constant term of $\rho_\fa$ with respect to $\tau$ and let $\phi_{\fa} \in \Gal(H/K)$ denote the Artin automorphism associated to $\fa$, and let the $B$ be the integral closure of $A$ in $H$.  We define a zeta function associated to $\rho$ twisted by the parameter $b\in B$ to be
\[
\zeta_\rho(b;s) := \sum_{\fa \subseteq A} \frac{b^{\phi_\fa}}{\pd(\rho_\fa)^s}.
\]
\begin{theorem:ZetaValues}\label{T:ZetaValues}
For $b\in B$ and for $n\leq q-1$, there exists a constant $C\in H$ and a vector $(*,\dots,*,C\zeta_\rho(b;n))^\top \in \C_\infty^n$ such that
\[\bd := \Exp_\rho\on\left (\begin{matrix}
*\\
\vdots\\
*\\
C\zeta_\rho(b;n)
\end{matrix}\right ) \in H^n,\]
where $C \in H$ and $\bd\in H^n$ are explicitly computable as outlined in the proof.
\end{theorem:ZetaValues}

In \S \ref{S:Transcendence} we discuss the transcendence implications of theorem \ref{T:ZetaValues}.  Using techniques similar to Yu's in \cite{Yu97} we prove the following theorem.

\begin{theorem:Transcendence}\label{T:Transcendence}
Let $\rho$ be a rank 1 sign-normalized Drinfeld module, let $\pi_\rho$ be a fundamental period of the exponential function associated to $\rho$ and define $\zeta_\rho(b;n)$ as above.  Then for $m\leq q-1$ 
\[\dim_{\overline{K}}\Span_{\overline{K}} \{\zeta_\rho(b;1),\dots,\zeta_\rho(b;m),1,\pi_\rho,\dots,\pi_\rho^{m-1} \} = 2m.\]
\end{theorem:Transcendence}

From Theorem \ref{T:Transcendence} we get a corollary which relates to a theorem of Goss (see \cite[Thm. 2.10]{Goss80}).

\begin{corollary:TranscedenceCor}\label{C:Transcendence}
For $1\leq i\leq q-1$, the quantities $\zeta_\rho(b;i)$ are transcendental.  Further, the ratio $\zeta_\rho(b;i)/\pi_\rho^j \in \overline K$ for $0\leq j\leq q-1$ if and only if $i=j=q-1$.
\end{corollary:TranscedenceCor}

Finally in \S \ref{S:Examples} we give examples of the constructions in our main theorems.

\begin{remark}
The author would like to thank his doctoral adviser Matt Papanikolas for many helpful discussions on the topics of this paper and for his continued support throughout his studies.
\end{remark}

\section{Background and notation}\label{S:Background}
As this paper builds on the foundation laid out in \cite{GreenPeriods17}, we require much of the same notation given there.  Let $q=p^r$ for a prime $p$ and an integer $r>0$.  Define the elliptic curve $E$ over $\F_q$, the finite field of size $q$, with Weierstrass equation
\begin{equation}\label{ecequation}
E: y^2 + c_1 ty + c_3y = t^3 + c_2 t^2 + c_4t + c_6,\quad c_i\in \F_q,
\end{equation}
and denote the point at infinity as $\infty$.

\begin{tablesymb}\label{tableofsymbols} \textup{We use the following symbols throughout the paper}\\
\begin{tabular}{l r l l}
$\bA$ &$=$ & $\F_q[t,y]$, &the affine coordinate ring of $E$\\
$\bK$ &$=$ & $\F_q(t,y)$, &the field of fractions of $\bA$\\
$\lambda$&$=$ &$\frac{dt}{2y + c_1t + c_3}$, &the invariant differential on $E$\\
$A$ &$=$ & $\F_q[\theta,\eta]$, &an isomorphic copy of $\bA$ with variables $\theta$, $\eta$\\
$K$ &$=$ & $\F_q(\theta,\eta)$ &an isomorphic copy of $\bK$ with variable $\theta$, $\eta$\\
$\ord_\infty$ &= &  &the valuation of $K$ (and $\bK$) at the infinity place\\
$\deg$ &$=$ & $-\ord_\infty$, &the degree function on $K$, normalized with $\lvert \theta\rvert = 2$ and $\lvert \eta\rvert = 3$\\
$\lvert\cdot\rvert$ &$=$ & $q^{\deg(\cdot)}$, &an absolute value on $K$\\
$K_\infty$ &$=$ &$\widehat K$  &the completion of $K$ at the infinite place\\
$\C_\infty$ &$=$ &$\widehat{\overline{K}}_\infty$  &the completion of an algebraic closure of $K_\infty$\\
$\Xi$ &$=$ &$(\theta,\eta)$, &a point on $E(K)$\\
\end{tabular}
\end{tablesymb}

Define canonical isomorphisms
\begin{equation}\label{canoniso}
\iota:\bK\to K, \quad \chi:K\to \bK
\end{equation}
such that $\iota(t)=\theta$ and $\iota(y) = \eta$ and similarly for $\chi$.  For ease of notation, for $x\in K$ we will sometimes refer to $\chi(x) = \overline x$, i.e. $\overline x$ denotes the element $x$ expressed with the variable $t$ and $y$.  We remark that the isomorphisms $\iota$ and $\chi$ extend to finite algebraic extensions of $\bK$ and $K$, and that $\ord_\infty$, $\deg$ and $\lvert\cdot\rvert$ extend to $K_\infty$ and $\C_\infty$.

Define a seminorm on $M = \langle m_{i,j} \rangle \in \Mat_{\ell\times m}(\C_\infty)$ which extends $\lvert\cdot\rvert$ as in \cite[\S 2.2]{PLogAlg}  by defining
\[|M| = \max_{i,j}(|m_{i,j}|).\]
Note that the seminorm is not multiplicative in general, but for matrices $M\in \Mat_{k\times \ell}(\C_\infty)$ and $N\in \Mat_{\ell\times m}(\C_\infty)$ we do have
\[|MN|\leq |M|\cdot |N|.\]
Also, for $c\in \C_\infty$ and $M,N \in \Mat_{\ell\times m}(\C_\infty)$ we have
\[|cM| = |c|\cdot |M|,\quad |M+N|\leq |M|+|N|.\]

Observe that $\bA$ has a basis $\bA = \Span_{\F_q}(t^i,t^jy),$ for $i,j\geq 0$ and that each term has unique degree.  Thus, when expressed in this basis, an element $a\in \bA$ has a leading term which allows us to define a sign function
\[\sgn:\bA\setminus \{0\} \to \F_q^\times,\]
by setting $\sgn(a)\in \F_q^\times$ to be the coefficient of the leading term of $a\in \bA\setminus \{0\}$.  We also define $\sgn$ on $\bK$, $A$ and $K$ in the natural way.  We extend $\sgn$ further, for any extension $L/\F_q$ using the same notion of leading term for the field $L(t,y)$, and we denote this extended sign function
\[\tsgn:L(t,y)^\times \to L^\times.\]

If $L/\F_q$ is an algebraically closed extension of fields, then we define $\tau:L\to L$ to be the $q$th power Frobenius map and $L[\tau]$ to be the ring of twisted polynomials in $\tau$, subject to the relation $\tau c = c^q\tau$ for $c\in L$.  Define the Frobenius twisting automorphism for $g = \sum c_{j,k} t^jy^k\in  L[t,y]$ to be
\begin{equation}\label{twistdef}
g\twist = \sum c_{j,k}^q t^jy^k,
\end{equation}
and let $g\twisti$ denote the $i$th iteration of twisting.  We extend twisting to matrices in $\Mat_{\ell\times m}(L(t,y))$ by twisting entry-wise and use this notion of twisting to define $\Mat_n(L)[\tau]$, the non-commutative ring of polynomials in $\tau$ subject to the relation $\tau M = M\twist \tau$ for $M\in \Mat_n(L)$.  In the setting of Anderson $\bA$-modules, we let $\Mat_n(L)[\tau]$ act on $L^n$ for $n\geq 1$ via twisting, i.e. for $\Delta = \sum M_i \tau^i$, with $M_i\in \Mat_n(L)$ and $\ba \in L^n$,
\begin{equation}\label{Deltaaction}
\Delta(\ba) = \sum M_i \ba\twisti.
\end{equation}
Further, for $X = (t_0,y_0)\in E(L)$, we define $X\twist = (t_0^q,y_0^q)$ and extend twisting to divisors in the obvious way, noting that for $g\in L(t,y)$
\[\divisor(g\twist) = \divisor(g)\twist.\]

For $c\in A$, define the Tate algebra 
\begin{equation} \label{Tatealgs}
   \TT_c = \biggl\{ \sum_{i=0}^\infty b_i t^i \in \power{\C_\infty}{t} \biggm| \big\lvert c^i b_i \big\rvert \to 0 \biggr\},
\end{equation}
the set of power series which converge on the closed disk of radius $|c|$.  For convenience, we set $\TT:= \TT_1$, and we have natural embeddings $\bA \hookrightarrow \TT_\theta[y] \hookrightarrow \TT[y]$.  For a fixed dimension $n>0$, define the Gauss norm $\lVert \cdot \rVert_c$ for a vector of functions $\bh = \sum \bd_i t^i \in \TT_c^n$ with $\bd_i\in \C_\infty^n$ by setting
\[\lVert \bh \rVert_c = \max_i |c^i\bd_i|,\]
where $|\cdot|$ is the matrix seminorm described above.  Extend $\lVert \cdot \rVert_c$ to $\TT_c[y]^n$ by setting $\lVert \bh_1+y\bh_2 \rVert_c = \max(\lVert \bh_1 \rVert_c ,\lVert \eta \bh_2 \rVert_c )$ for $\bh_1,\bh_2\in \TT_c^n$.  Note that $\TT_c[y]^n$ is complete under the Gauss norm.  Using the definition given in \cite[Chs.~3--4]{FresnelvdPut}, we observe that the rings $\TT[y]$ and $\TT_\theta[y]$ are affinoid algebras corresponding to rigid analytic affinoid subspaces of $E/\C_\infty$.  If we denote $\cE$ as the rigid analytic variety associated to $E$ and $\cU\subset \cE$ as the inverse image under $t$ of the closed disk in $\C_\infty$ of radius $|\theta|$ centered at 0, then $\cU$ is the affinoid subvariety of $\cE$ associated to $\TT_\theta[y]$.  Note that Frobenius twisting extends to $\TT_c[y]^n$ and its fraction field and that $\TT$ and $\TT[y]$ have $\F_q[t]$ and $\bA$, respectively, as their fixed rings under twisting (see \cite[Lem.~3.3.2]{P08}).  We extend the action of $\Mat_n(L)[\tau]$ on $L^n$ described in \eqref{Deltaaction} to an action of $\Mat_n(\TT[y])[\tau]$ on $\TT[y]^n$ in the natural way.

\section{Review of tensor powers of Drinfeld modules and Anderson generating functions}\label{S:Review}
We recall several facts about rank 1 sign-normalized Drinfeld modules as set out in \cite[\S 3]{GP16} (see also \cite{Goss}, \cite{Hayes92} or \cite{Thakur} for a thorough account of Drinfeld modules).  First note that we can pick a unique point $V$ in $E(H)$ whose coordinates have positive degree such that $V$ satisfies the equation on $E$
\begin{equation}\label{Vequation}
(1-\Fr)(V) = V-V\twist = \Xi,
\end{equation}
If we set $V = (\alpha,\beta)$, then $\deg(\alpha) = 2$, $\deg(\beta) = 3$, $\sgn(\alpha) = \sgn(\beta) = 1$ and $H = K(\alpha,\beta)$.  There is a unique function in $H(t,y)$, called the shtuka function, with $\tsgn(f) = 1$ and with divisor
\begin{equation}\label{fdivisor}
\divisor(f) = (V\twist) - (V) + (\Xi) - (\infty).
\end{equation}
We can write
\begin{equation} \label{fdef}
  f = \frac{\nu(t,y)}{\delta(t)} = \frac{y - \eta - m(t-\theta)}{t-\alpha} = \frac{y + \beta + c_1\alpha + c_3 - m(t-\alpha)}{t-\alpha},
\end{equation}
where $m\in H$ is the slope between the collinear points $V\twist, -V$ and $\Xi$, and $\deg(m) = q$, and
\begin{gather} \label{nudiv}
\divisor(\nu) = (V\twist) + (-V) + (\Xi) - 3(\infty), \quad \divisor(\delta) = (V) + (-V) - 2(\infty).
\end{gather}
Let $L/K$ be an algebraically closed field.  A Drinfeld $\bA$-module over $L$ is an $\F_q$-algebra homomorphism $\rho : \bA \to L[\tau]$ such that for all $a \in \bA$,
\[
  \rho_a = \iota(a) + b_1 \tau + \dots + b_n \tau^n.
\]
The rank $r$ of $\rho$ is the unique integer such that $n = r \deg a$ for all $a$.  Rank 1 sign-normalized means that we require $r=1$ and that $b_n = \sgn(a)$.

For a Drinfeld $\bA$-module $\rho$, we denote the exponential and logarithm function as 
\[
  \exp_\rho(z) = \sum_{i=0}^\infty \frac{z^{q^i}}{d_i},\quad   \log_\rho(z) = \sum_{i=0}^\infty \frac{z^{q^i}}{\ell_i} \in \power{H}{z}, \quad d_0 = \ell_0=1.
\]
Formulas for the coefficients of $\exp_\rho$ and $\log_\rho$ are given in \cite[Thm. 3.4 and Cor. 3.5]{GP16} as
\begin{equation}\label{Drinfeldexp}
  \exp_\rho(z) = \sum_{i=0}^\infty \frac{z^{q^i}}{(f f^{(1)} \cdots f^{(i-1)})|_{\Xi^{(i)}}},
\end{equation}
\begin{equation}\label{Drinfeldlog}
  \log_\rho(z) = \sum_{i=0}^\infty \Res_{\Xi} \biggl( \frac{ \tlambda^{(i+1)}}{f f^{(1)} \cdots f^{(i)}} \biggr) z^{q^i} =  \sum_{i=0}^\infty \left ( \frac{ \delta^{(i+1)}}{\delta^{(1)} f^{(1)} \cdots f^{(i)}} \bigg|_{\Xi} \right ) z^{q^i},
\end{equation}
where $\tlambda \in \Omega^1_{E/H}(-(V) + 2(\infty))$ is the unique differential $1$-form such that $\Res_{\Xi}(\tlambda^{(1)} / f) = 1$.

We now recount the theory of $n$-dimensional tensor powers of $\bA$-motives and dual $\bA$-motives from \cite[\S 3-4]{GreenPeriods17}.  For $n\geq 1$, let $U = \Spec L[t,y]$ be the affine curve $(L \times_{\F_q} E) \setminus \{ \infty \}$.  Define the underlying space of the $\bA$-motive $M$ and the dual $\bA$-motive $N$ as
\begin{equation}\label{MNdef}
M = \Gamma (U,\cO_E(nV)),\quad N = \Gamma (U,\cO_E(-nV\twist)),
\end{equation}
define an $L[t,y,\tau]$-action on $M$ and an $L[t,y,\sigma]$-action on $N$ by letting $L[t,y]$ act by multiplication and defining the action for $a\in M$ and $b \in N$ as
\begin{equation}\label{tauaction}
\tau a = f^na\twist\quad \text{and}\quad \sigma b = f^n b\twistinv.
\end{equation}
We remark that $M$ and $N$ are the $n$th tensor powers of an $\bA$-motive and a dual $\bA$-motive respectively, and we refer the reader to \cite[\S 3]{GreenPeriods17} for details on this construction.

Define a functions $g_i \in M$ for $1\leq i\leq n$ with divisors
\begin{equation}\label{gidivisor}
\divisor(g_{j}) = -n(V) +  (n-j)(\infty) + (j-1)(\Xi) + ([j-1]V\twist + [n-(j-1)]V),
\end{equation}
and functions $h_i \in N$ with divisors
\begin{equation}\label{hidivisor}
\divisor(h_j) = n(V\twist) - (n+j)(\infty) + (j-1)(\Xi) + (-[n - (j - 1)]V\twist-[j-1]V),
\end{equation}
with $\tsgn(g_i) = \tsgn(h_i)=1$.  When it is convenient, we will extend the definitions of the functions $g_i$ and $h_i$ for $i>n$ by writing $i=jn+k$, where $1\leq k< n$, and then denoting,
\begin{equation}\label{gihigher}
g_i := \tau^j(g_k) = (ff\twist \dots f\twistk{j-1})^n g_k\twistk{j}\quad \text{and}\quad h_i := \sigma^j(h_k) = (ff\twistinv \dots f\twistk{1-j})^n h_k\twistk{-j}.
\end{equation}
\begin{proposition}\label{P:review}
The following facts about the functions $g_i$ and $h_i$ are proved in \cite[\S 3-4]{GreenPeriods17}:
\begin{enumerate}
\item[(a)] For $n\geq 2$, the set of functions $\{g_i\}_{i=1}^n$ generate $M$ as a free $L[\tau]$-module and the set of functions $\{h_i\}_{i=1}^n$ generate $N$ as a free $L[\sigma]$-module.
\item[(b)] For $1\leq j\leq n-1$ we obtain the following identities of functions
\[g_1h_1\twistinv = t-t([n]V),\]
\[g_{j+1}h_{n-(j-1)} = f^n\cdot (t-t([j]V\twist+[n-j]V)).\]
\item[(c)] For $1\leq k\leq n$, the quotient functions $(g_{k+1}/g_{k})$ have divisors 
\[
\divisor(g_{k+1}/g_{k}) = (\Xi) - (\infty) + ([k]V\twist + [n - k]V) - ([k-1]V\twist + [n - (k-1)]V).
\]
\item[(d)] We can write $(g_{k+1}/g_{k})$ as a quotient of a linear function of degree 3 and a linear function of degree 2, which we label
\[
\frac{\nu_k(t,y)}{\delta_k(t)} := \frac{y - \eta - m_k(t-\theta)}{t - t([k-1]V\twist + [n - (k-1)]V)} =\frac{g_{k+1}}{g_k}  ,
\]
for $1\leq k\leq n$, where $m_k$ is the slope between the points $[k]V\twist + [n - k]V$ and $[-(k-1)]V\twist - [n - (k-1)]V$.
\item[(e)] For $1\leq i\leq n$, there exist constants $a_i,b_i,y_i,z_i\in H$ such that we can write 
\begin{align*}
tg_i &= \theta g_i + a_ig_{i+1} + g_{i+2},\\
y g_i & = \eta g_i + y_i g_{i+1} + z_i g_{i+2} + g_{i+3},\\
th_i &= \theta h_i + b_ih_{i+1} + h_{i+2}.
\end{align*}
\item[(f)] For the constants defined in \textnormal{(e)} we have $a_j = b_{n-j}$ for $1\leq j\leq n-1$ and $a_n = b_n^q.$
\item[(g)] The coefficients $a_i$ are given by
\[a_i = \frac{2\eta + c_1\theta + c_3}{\theta - t([i]V\twist + [n-i]V)}\]
\end{enumerate}
\end{proposition}

An $n$-dimensional Anderson $\bA$-module is a map $\rho:\bA\to \Mat_n(L)[\tau]$, such that for $a\in \bA$
\[\rho_a = d[a] + A_1 \tau + \dots + A_m\tau^m,\]
where $d[a] = \iota(a)I + N$ for some nilpotent matrix $N\in \Mat_n(L)$.  We will always label the constant coefficient of $\rho_a$ as $d[a]$, and we remark that $d:\bA\to \Mat_n(L)$ is a ring homomorphism.  The map $\rho$ describes an action of $\bA$ on the underlying space $L^n$ in the sense defined in \eqref{Deltaaction}, allowing us to view $L^n$ as an $\bA$-module.  In what follows, for convenience, we fix the algebraically closed field $L$ to be $\C_\infty$, although we remark that much of the theory applies equally to any algebraically closed field.

To ease notation throughout the paper, for a fixed dimension $n$, we define $N_i\in \Mat_n(\F_q)$ for an integer $i\geq 1$ to be the matrix with $1$'s along the $i$th super-diagonal and $0$'s elsewhere and define $N_i$ for $i\leq -1$ to be the matrix with $1$'s along the $i$th sub-diagonal and $0$'s elsewhere.  We also define $E_1$ to be the matrix with a single $1$ in the lower left corner and zeros elsewhere and in general define $E_i$ to be $N_{i-n}$.  We also define $N_{i}(\alpha_1,\dots,\alpha_{n-i})$ to be the matrix with the entries $\alpha_1,\alpha_2, \dots, \alpha_{n-i}$ along the $i$th super diagonal and similarly for $N_{i-n}(\alpha_1,\dots,\alpha_{n-i})$ and $E_i(\alpha_1,\dots, \alpha_i)$.

Given $\bA$, an affine coordinate ring of an elliptic curve, \cite[\S 3]{GP16} describes how to construct $\rho$, the unique sign-normalized rank 1 Drinfeld module associated to $\bA$.  Then \cite[\S 4]{GreenPeriods17} describes how to construct the $n$th tensor power of $\rho$ by setting
\begin{equation}\label{taction}
\rho\on_t := d[\theta] + E_\theta \tau := (\theta I +  N_1(a_1,\dots,a_{n-1}) + N_2) + (E_1(a_n)+E_2)\tau,
\end{equation}
\begin{align}\label{yaction}
\begin{split}
\rho\on_y := d[\eta] + E_\eta \tau:= &(\theta I +  N_1(y_1,\dots,y_{n-1}) + N_2(z_1,\dots,z_{n-2}) + N_3)  \\
&+(E_1(y_n)+E_2(z_{n-1},z_n) + E_3)\tau,
\end{split}
\end{align}
where $a_i$, $y_i$ and $z_i$ are given in Proposition \ref{P:review}.

To simplify notation later, we define strictly upper triangular matrices
\begin{equation}\label{Nthetadef}
N_\theta = d[\theta] - \theta I\quad\text{and}\quad N_\eta = d[\eta] -   \eta I .
\end{equation}
With the definitions of $\rho\on_t$ and $\rho\on_y$, we define the $\F_q$-linear map
\[\rho\on_a: \bA \to \Mat_n(H)[\tau]\]
for any $a\in \bA$ by writing $a=\sum c_it^i + y\sum d_it^i$ with $c_i,d_i\in \F_q$, and extending using linearity and the composition of maps $\rho\on_{t^a} = (\rho\on_t)^a$.  A priori, the map $\rho\on$ is just an $\F_q$-linear map, but using ideas from \cite{HJ16} the author proves in \cite{GreenPeriods17} that $\rho\on$ is actually an Anderson $\bA$-module.

We will label the exponential and logarithm function associated to $\rho\on$ as
\begin{equation}\label{Exptensor}
\Exp\on_\rho(\bz) = \sum_{i=0}^\infty Q_i \bz\twisti\in \power{\Mat_n(\C_\infty)}{\bz},\quad \Log_\rho\on(\bz) = \sum_{i=0}^\infty P_i \bz\twisti \in \power{\Mat_n(\C_\infty)}{\bz},
\end{equation}
defined so that $Q_0 = P_0 = I$.  We note that $\Log_\rho\on$ is defined to be the formal inverse of the power series $\Exp_\rho\on$, and that the exponential and logarithm functions satisfy functional equations for all $a\in \bA$ and $\bz\in \C_\infty^n$
\begin{equation}\label{expfunctionalequation}
\Exp_\rho\on(d[a]\bz) = \rho\on_a(\Exp_\rho\on(\bz)), \quad \Log _\rho(\rho\on _a(\bz)) = d[a]\Log _\rho(\bz).
\end{equation}
We also note that $\Exp_\rho\on$ is an entire function from $\C_\infty^n$ to $\C_\infty^n$ and that $\Log _\rho$ has a finite radius of convergence in $\C_\infty^n$ which we denote $r_L$.

We now recall facts about the spaces $\Omega$ and $\Omega_0$ and about vector-valued Anderson generating functions from \cite[\S 5-6]{GreenPeriods17}.  For $n>1$ define the space of rigid analytic functions
\[\BB := \Gamma \bigl( \cU, \cO_E(-n(V) + n(\Xi))\bigr)\]
where $\cU$ is the inverse image under $t$ of the closed disk in $\C_\infty$ of radius $|\theta|$ centered at 0 defined in section \S \ref{S:Background} and define $\bA$-modules of functions
\begin{equation}\label{Omegadef}
\Omega = \{h\in \BB\mid h\twist - f^n h = g \in N\},\quad \Omega_0 = \{h\in \BB\mid h\twist - f^n h = 0\},
\end{equation}
where we recall the dual $\bA$-motive $N=\Gamma(U,\cO_E(-nV\twist))$.  For a function $h(t,y)\in \Omega$, define the map $T:\Omega \to \TT[y]^n$ by
\begin{equation}\label{Tmapdef}
T( h(t,y)) = (h(t,y)g_1, \dots , h(t,y) g_n)^\top,
\end{equation}
where the functions $g_i$ are the basis elements defined in Proposition \ref{P:review}.  For ease of notation later on, we also define
\begin{equation}\label{gvector}
\bg := (g_1,\dots,g_n)^\top.
\end{equation}

Define operators on the space $\TT[y]^n$ which act in the sense defined in \S \ref{S:Background} by setting
\begin{equation}
D_t := \rho\on_t - t,\quad\text{and}\quad D_y = \rho\on_y - y,
\end{equation}
\begin{equation}\label{GE1operator}
G - E_1\tau := (\text{diag}(g_2/g_1,\dots, g_{n+1}/g_n) - N_1) - E_1\tau.
\end{equation}
A quick calculation shows that $\left (G-E_1\tau\right )(T(h)) = 0$ for any $h\in \Omega_0$, and thus the operator $G-E_1\tau$ can be viewed as a vector version of the operator $\tau - f^n$.  In fact, the relationship is even stronger, as is proved in the following lemma.

\begin{lemma}[Lemma 5.3 of \cite{GreenPeriods17}]\label{L:GE1andOmega}A vector $J(t,y)\in \TT[y]^n$ satisfies $(G-E_1\tau)(J)=0$ if and only if there exists some function $h(t,y)\in \Omega_0$ such that $J(t,y) = T(h(t,y))$.
\end{lemma}

Define the operator $M_\tau :=N_1 + E_1\tau,$ and denote the diagonal matrices
\begin{equation}\label{Mm}
M_m := \text{diag}(z_1-a_2,z_2-a_3, \dots , z_{n-1} - a_n, z_n - a_1\twist),\quad M_\delta:= \text{diag}(\delta_1,\delta_2, \dots , \delta_n).
\end{equation}
where $a_i$, $z_i$ and $\delta_i$ are defined in Proposition \ref{P:review}.  Then for $1\leq i\leq n$ denote
\[p_i = \eta - y - (\theta - t)(z_i - a_{i+1}),\quad r_i = y_i - (\theta - t) - a_i(z_i - a_{i+1})\]
where we understand $a_{n+1}$ to be $a_1\twist$, and define matrices
\begin{equation}\label{M'def}
M_1' = \text{diag}(p_1,\dots,p_n) + N_1(r_1,\dots r_{n-1}),\quad M_2' = E_1(r_n).
\end{equation}
Also define matrices
\begin{equation}\label{Midef}
M_1 = M_1'\big|_{t=0,y=0}\quad \text{and}\quad M_2 = M_2'\big|_{t=0,y=0},
\end{equation}
where above we formally evaluate $M_i'$ at $t=0$ and $y=0$.
\begin{proposition}\label{P:OperatorReview}
We have the following facts from \cite[\S 5]{GreenPeriods17} about the above operators:
\begin{enumerate}
\item[(a)] $(G - E_1\tau) = M_\delta\inv (D_y - (M_\tau + M_m) D_t)$
\item[(b)] $M_1'\bg + M_2' f^n \bg\twist = 0$
\item[(c)] $M_1 + M_2\tau = (\rho\on_y - (M_\tau + M_m) \rho\on_t)$
\end{enumerate}
\end{proposition}

We now recall the functions $\omega_\rho$, $E_\bu\on$ and $G_\bu\on$ defined in \cite[\S 4]{GP16}.  Let
\begin{equation} \label{omegarhoprod}
  \omega_\rho = \xi^{1/(q-1)} \prod_{i=0}^\infty \frac{\xi^{q^i}}{f^{(i)}}, \quad \xi = -\frac{m\theta -\eta}{\alpha} = -\biggl( m + \frac{\beta}{\alpha} \biggr),
\end{equation}
where $m$, $\alpha$, and $\beta$ are given at the beginning of this section and recall that $\omega_\rho \in \TT[y]^\times$ and $\omega_\rho^n \in \Omega_0$.  For $\bu = (u_1,...,u_n)^\top \in \C_\infty^n$ define
\begin{equation}\label{Eudef}
E_{\bu}^{\otimes n}(t) =  \sum_{i=0}^\infty \Exp_\rho\on\left (d[\theta]^{-i-1} \bu\right ) t^i,
\end{equation}
\begin{equation}\label{Gudef}
G_\bu\on (t,y) = E_{d[\eta]\bu}\on(t) + (y+c_1t + c_3) E_\bu\on(t).
\end{equation}  We remark that 1-dimensional Anderson generating functions have proved useful in studying algebraic relations among logarithm values, periods, quasi-periods, $L$-series and motivic Galois groups of Drinfeld modules (e.g., see \cite{CP11}, \cite{CP12}, \cite{EP14}, \cite{Pellarin08}--\cite{Perkins14a}, \cite{Sinha97}).  In this paper we use vector-valued Anderson generating functions to get formulas for the coefficients of the exponential and logarithm functions.

\begin{proposition} \label{P:EuReview}
We collect the following facts from \cite[\S 5-6]{GreenPeriods17} about the above functions:
\begin{enumerate}
\item[(a)] The function $\omega_\rho^n$ generates $\Omega_0$ as a free $\bA$-module.
\item[(b)] The function $E_\bu\on \in \TT^n$ and we have the following identity of functions in $\TT^n$
\[E_\bu\on(t) = \sum_{j=0}^\infty Q_j\left (d[\theta]\twistk{j} -   tI \right )\inv\bu\twistk{j}.\]
where $Q_i$ are the coefficients of $\Exp_\rho\on$ from \eqref{Exptensor}.
\item[(c)] The function $G_\bu\on$ extends to a meromorphic function on $U = (\C_\infty \times_{\F_q} E) \setminus \{ \infty \}$ with poles in each coordinate only at the points $\Xi\twisti$ for $i\geq 0$.
\item[(d)] The operators $D_t$ and $D_y$ acting on $G_\bu\on$ give
\begin{align*}
D_t(G_\bu\on) &= \Exp_\rho\on(d[\eta]\bu) + (y+c_1t + c_3)\Exp_\rho\on(\bu)\\
D_y(G_\bu\on) &= -c_1 \Exp_\rho\on(d[\eta] \bu) + \Exp_\rho\on(d[\theta^2] \bu)+ (t+c_2)\Exp_\rho\on(d[\theta] \bu)\\
&  + (t^2 + c_2 t + c_4) \Exp_\rho\on(\bu).
\end{align*}
\end{enumerate}
\end{proposition}

Define $\cM$ to be the submodule of $\TT[y]$ consisting of all elements in $\TT[y]$ which have a meromorphic continuation to all of $U$.  Now define the map $\RES_{\Xi}:\cM^n\to \C_\infty^n,$ for a vector of functions $(z_1,...,z_n)^\top\in \cM^n$ as
\begin{equation}\label{RESdef}
\RES_{\Xi}((z_1,\dots, z_n)^\top) = (\Res_{\Xi}(z_1\lambda),\dots, \Res_{\Xi}(z_n\lambda))^\top
\end{equation}
where $\lambda$ is the invariant differential of $E$ from \eqref{tableofsymbols}.

\begin{proposition}\label{P:ResidueReview}
We recall the following facts about the map $\RES_\Xi$ from \cite[\S 6]{GreenPeriods17}:
\begin{enumerate}
\item[(a)] $\RES_\Xi(G_\bu\on) =  -(u_1,\dots,u_n )^\top$
\item[(b)] If we denote $\Pi_n = -\RES_\Xi(T(\omega_\rho^n)),$ then $T(\omega_\rho^n) = G_{\Pi_n}\on$ and the period lattice of $\Exp_\rho\on$ equals $\Lambda_\rho\on = \{d[a]\Pi_n\mid a\in \bA\}$.
\item[(c)] If $\pi_\rho$ is a fundamental period of the exponential function $\exp_\rho$ from \eqref{Drinfeldexp}, and if we denote the last coordinate of $\Pi_n \in \C_\infty^n$ as $p_n$, then $p_n/\pi_\rho^n \in H$.
\end{enumerate}
\end{proposition}

\section{Coefficients of the exponential function}\label{S:ExpCoef}
The coefficients of the exponential function for rank 1 sign-normalized Drinfeld modules are well understood (see \eqref{Drinfeldexp}).  Further, the coefficients for the exponential function of the $n$th tensor power of the Carlitz module are also well understood.  These coefficients were first studied by Anderson and Thakur in \cite[\S 2.2]{AndThak90}, and have recently been written down explicitly using hyper derivatives by Papanikolas in \cite[4.3.6]{PLogAlg}.  In this section we give explicit formulas for the coefficients of the exponential function for the $n$th tensor power of a rank 1 sign-normalized Drinfeld module.

In order to write down a formula for the coefficients of $\Exp_\rho\on$ we must first analyze certain functions which arise when calculating residues of the vector-valued Anderson generating functions $G_\bu\on$.  For a fixed dimension $n$, for $1\leq \ell\leq n$ and for $i\geq 0$, define the functions
\begin{equation}\label{gammadef}
\gamma_{i,\ell} = \frac{g_\ell}{(ff\twist \dots f\twistk{i-1})^n},
\end{equation}
where for $i=0$ we understand $\gamma_{0,\ell} = g_\ell$.  Using \eqref{fdivisor} and \eqref{gidivisor} we see that the polar part of the divisor of $\gamma_{i,\ell}$ equals
\[ -n(V\twisti)   -n(\Xi\twistk{i-1}) - n(\Xi\twistk{i-2}) - \dots - (n-(\ell-1))(\Xi). \]
We temporarily fix an index $\ell$.  Using the Riemann-Roch theorem, we observe that we can find unique functions $\alpha_{j,k}$ with $\tsgn(\alpha_{j,k})=1$ in each of the following 1-dimensional spaces, for $1\leq j\leq i$ and $1\leq k\leq n$,
\[\alpha_{j,k} \in \mathcal L(n(V\twisti) - n(\Xi\twisti) + k(\Xi\twistk{j-1}) +n(\Xi\twistk{j-2}) + \dots+ n(\Xi\twist) + n(\Xi)-(n(j-1)+k-1)(\infty)).\]
Then, for appropriate constants $d_{j,k}\in H$ we subtract off the principal part of the power series expansion of $g_\ell/(ff\twist \dots f\twistk{i-1})^n$ at $\Xi\twistk{m}$, for $1\leq m \leq j-1$, to find that
\[\gamma_{i,\ell} - \sum_{j,k} d_{j,k} \alpha_{j,k} \in \mathcal L\left (n(V\twisti)\right ) = \Span_{H}(g_1\twisti, g_2\twisti, \dots , g_n\twisti).\]
So for further constants $c_{\ell,1},\cdots, c_{\ell,n} \in H$ we can write
\begin{equation}\label{gammaexpression}
\gamma_{i,\ell} = c_{\ell,1}g_1\twisti +c_{\ell,2}g_2\twisti + \dots c_{\ell,n}g_n\twisti +  \sum_{j,k} d_{j,k} \alpha_{j,k},
\end{equation}
where we note that each of the functions $\alpha_{j,k}$ vanishes with order $n$ at $(\Xi\twisti)$ and that the coefficients $c_{\ell,k}$ are implicitly dependent on $i$.  To ease notation, for each $1\leq \ell\leq n$ we will write $\alpha_\ell := \sum_{j,k} d_{j,k} \alpha_{j,k}$ and denote
\begin{equation}\label{Ciequation}
\bgamma_i = (\gamma_{i,1}, \gamma_{i,2},\dots,\gamma_{i,n})^\top, \quad C_i = \langle c_{j,k}\rangle,\quad \text{and}\quad \balpha_i =(\alpha_1, \alpha_2,\dots,\alpha_n)^\top,
\end{equation}
so that we can write equation \eqref{gammaexpression} for $1\leq \ell\leq n$ as $\bgamma_i = C_i \bg\twisti + \balpha_i.$

\begin{theorem}\label{T:ExpCoef}
With the notation as above, for dimension $n\geq 2$ and $\bz\in \C_\infty$, if we write
\[\Exp_\rho\on(\bz) = \sum_{i=0}^\infty Q_i \bz\twisti,\]
then for $i\geq 0$, the exponential coefficients $Q_i = C_i$ and $Q_i\in \Mat_n(H)$.
\end{theorem}
\begin{remark}
We remark that in the case for $n=1$, if one interprets the empty divisors in \eqref{gidivisor} correctly, then Theorem \ref{T:ExpCoef} still holds.  However, for clarity of exposition, we restrict to $n\geq 2$.
\end{remark}
Before giving the proof of Theorem \ref{T:ExpCoef}, we require a lemma about the coefficients of the exponential function.
\begin{lemma}\label{L:Recurrencelemma}
Given a sequence of matrices $Q_i\in \text{Mat}_n(H)$ for $i\geq 0$ with $Q_0=I$, then the $Q_i$ are the coefficients of $\Exp_\rho\on$ if and only if they satisfy the recurrence relation for $i\geq 1$
\begin{equation}\label{recurrence}
M_2 Q_{i-1}\twist + E_1 Q_{i-1}\twist d[\theta]\twisti = Q_{i} d[\eta]\twistk{i} - (N_1 + M_m)Q_{i} d[\theta]\twistk{i} - M_1Q_{i}.
\end{equation}
where $M_1$ and $M_2$ are defined in \eqref{Midef} and $M_m$ is defined in \eqref{Mm}.  Further, the coefficients $Q_i \in \Mat_n(H)$.
\end{lemma}
\begin{proof}
First note that by \eqref{expfunctionalequation}
\[(\rho\on_y - (M_\tau+M_m) \rho\on_t)(\Exp_\rho\on(z)) = \Exp_\rho\on(d[\eta] z) - (M_\tau + M_m)\Exp_\rho\on(d[\theta] z).\]
Then, using Proposition \ref{P:OperatorReview}(c),
\[(M_1 + M_2\tau)(\Exp_\rho\on(z)) = \Exp_\rho\on(d[\eta] z) - (M_\tau + M_m)\Exp_\rho\on(d[\theta] z),\]
and expanding $\Exp_\rho\on$ on both sides in terms of its coefficients $Q_i$ and equating like terms gives the equality
\[M_2 Q_{i-1}\twist + E_1 Q_{i-1}\twist d[\theta]\twisti = Q_{i} d[\eta]\twistk{i} - (N_1 + M_m)Q_{i} d[\theta]\twistk{i} - M_1Q_{i}.\]
Thus the coefficients of the exponential function satisfy the recurrence relation \eqref{recurrence}.  Next, for $j\geq 0$, let $\{Q_j'\}\subset \Mat_n(H)$ be a sequence of matrices satisfying recurrence relation \eqref{recurrence}.  We will show that $\{Q_j'\}$ is uniquely determined by $Q_0$, and thus if we fix $Q_0=I$, the matrices $\{Q_j'\}$ will be the coefficients of $\Exp_\rho\on$.  Given a term $Q_{i-1}'$ of the sequence $\{Q_j'\}$ for $i\geq 1$, define
\[W_i = M_2 (Q_{i-1}')\twist + E_1(Q_{i-1}')\twist d[\theta]\twisti,\]
so that by \eqref{recurrence}
\begin{equation}\label{Widef}
W_i = Q_{i}'d[\eta]\twisti - (M_m+N_1)Q_{i}'d[\theta]\twisti - M_1 Q_i'.
\end{equation}
Then, denote $M_1 = M_d + M_n,$ where $M_d$ is the diagonal part of $M_1$ and $M_n$ is the nilpotent (super-diagonal) part.  Then collect the diagonal and off-diagonal terms of \eqref{Widef} to obtain
\begin{equation}\label{Wiequation}
W_i = (\eta^{q^i} I - \theta^{q^i}M_m  - M_d)Q_i' + Q_i' N_\eta \twisti - \theta^{q^i}N_1 Q_i' - M_m Q_i' N_\theta \twisti - N_1 Q_i' N_\theta\twisti - M_n Q_i',
\end{equation}
where we recall the definition of $N_\theta$ and $N_\eta$ from \eqref{Nthetadef}.  Next, we denote the matrix $M_D = \eta^{q^i} I - \theta^{q^i}M_m  - M_d,$ and note that it is diagonal and invertible.  Define
\[\beta_i:\Mat_n(H)\to \Mat_n(H)\]
to be the $\F_q$-linear map given for $Y\in \Mat_n(H)$ by
\begin{equation}\label{betamap}
Y\mapsto M_D\inv(YN_\eta\twisti - \theta^{q^i} N_1 Y - M_m Y N_\theta \twisti - N_1YN_\theta\twisti - M_nY).
\end{equation}
Note that $\beta_i$ is a nilpotent map with order at most $2n-1$, since matrix in definition \eqref{betamap}, except $M_D$, is strictly upper triangular, and thus each term of $\beta_i^{2n-1}$ will have at least $n$ strictly upper triangular matrices on either the left or the right of each matrix $Y$.  Then, using the map $\beta_i$ and rearranging slightly we can rewrite \eqref{Wiequation} as
\begin{equation}\label{tel1}
Q_i'+\beta_i(Q_i') = M_D\inv W_i.
\end{equation}
Applying $\beta_i^j$ to $\eqref{tel1}$, multiplying by $(-1)^j$, then adding these together for $j\geq 1$ gives a telescoping sum.  Since $\beta_i$ is nilpotent with order at most $2n-1$, we find
\begin{equation}\label{Qiexpression}
Q_i' = \sum_{j=0}^{2n-1} (-1)^j\beta_i^j(M_D\inv W_i).
\end{equation}
Thus we have determined $Q_i'$ uniquely in terms of $Q_{i-1}'$, and so each element in the sequence $\{Q_j'\}$ is determined by $Q_0$.  If we require that $Q_0=I$, then the matrices $\{Q_j'\}$ are the coefficients of $\Exp_\rho\on$.  Further, since $M_D$ and each matrix in the definition of $\beta_i$ is in $\Mat_n(H)$, we see that the exponential function coefficients $Q_i \in \Mat_n(H)$.
\end{proof}

We now return to the proof of Theorem \eqref{T:ExpCoef}.
\begin{proof}[Proof of Theorem \eqref{T:ExpCoef}]
We first recall that $\gamma_{0,\ell} = g_\ell$ and hence by \eqref{gammaexpression} we have $C_0 = I=Q_0$, so that the theorem is true trivially for $i=0$.  We then show that the sequence of matrices $\{C_i\}$ satisfies the recurrence in Lemma \ref{L:Recurrencelemma} for $i\geq 1$.  First observe that by Proposition \ref{P:review}(e)
\begin{equation}\label{dthetaaction}
d[\theta]\bg = t\bg - f^n E_\theta \bg\twist \quad \text{and}\quad d[\eta]\bg = y\bg - f^nE_\eta \bg\twist,
\end{equation}
with $\bg$ defined as in \eqref{gvector}.  Using \eqref{dthetaaction}, we write
\begin{align*}\label{recurrencecal}
&\left (M_2 C_{i-1}\twist + E_1 C_{i-1}\twist d[\theta]\twisti - C_{i} d[\eta]\twistk{i} + (N_1 + M_m)C_{i} d[\theta]\twistk{i} + M_1C_{i}\right )\bg\twisti\\
& = \left (M_2 C_{i-1}\twist  + tE_1 C_{i-1}\twist  - yC_{i} + t(N_1 + M_m)C_{i}  + M_1C_{i}\right ) \bg\twisti  \\
&\quad -\left (E_1 C_{i-1}\twist E_\theta\twisti - C_{i} E_\eta\twisti + (N_1 + M_m)C_{i}  E_\theta\twisti \right  ) f^n \bg\twisti.
\end{align*}
We examine the first term in the right hand side of the above equation, which we denote
\begin{equation}\label{firsttermvanishes}
T_1 = \left (M_2 C_{i-1}\twist  + tE_1 C_{i-1}\twist  - yC_{i} + t(N_1 + M_m)C_{i}  + M_1C_{i}\right ) \bg\twisti,
\end{equation}
and the second term, which we denote
\begin{equation}\label{secondterm}
T_2 = \left (E_1 C_{i-1}\twist E_\theta\twisti - C_{i} E_\eta\twisti + (N_1 + M_m)C_{i}  E_\theta\twisti \right  ) f^n \bg\twisti,
\end{equation}
separately.  By the discussion immediately following \eqref{Ciequation} we see that \eqref{firsttermvanishes} equals
\begin{align*}
T_1 &= (M_2+tE_1)\bgamma_{i-1}\twist + (-yI + t(M_m + N_1)+ M_1)\bgamma_i + \balpha_{i-1}\twist + \balpha_i\\
&= M_2'\bgamma_{i-1}\twist + M_1'\bgamma_i + \balpha_{i-1}\twist + \balpha_i,
\end{align*}
with $M_1'$ and $M_2'$ as given in \eqref{M'def}.  Then, writing out the coordinates of $\bgamma$ using the functions $\gamma_{i,\ell}$ from \eqref{gammadef} and finding a common denominator gives
\[T_1 = \frac{1}{(ff\twist \dots f\twistk{i-1})^n}\left (M_1'\bg + M_2' f^n \bg\twist + \balpha_{i-1}\twist + \balpha_i\right ) = \frac{1}{(ff\twist \dots f\twistk{i-1})^n}\left ( \balpha_{i-1}\twist + \balpha_i\right ),\]
since $M_1'\bg + M_2' f^n \bg\twist=0$ by Proposition \ref{P:OperatorReview}(b).  Thus $T_1$ vanishes coordinate-wise with order at least $n$ at $\Xi\twisti$, because the functions $\alpha_{\ell}$ from \eqref{Ciequation} each vanish with order at least $n$ at $\Xi\twisti$.  Further, the presence of the factored-out $f^n\bg\twisti$ shows that $T_2$ from \eqref{secondterm} also vanishes coordinate-wise with order at least $n$ at $\Xi\twisti$.  Thus we see that
\[\left (M_2 C_{i-1}\twist + E_1 C_{i-1}\twist d[\theta]\twisti - C_{i} d[\eta]\twistk{i} + (N_1 + M_m)C_{i} d[\theta]\twistk{i} + M_1C_{i}\right )\bg\twisti\]
consists of a constant matrix in $\Mat_n(H)$ multiplied by $\bg\twisti$, and equals a vector of functions which vanishes coordinate-wise with order at least $n$ at $\Xi\twisti$.  However, recall from \eqref{gidivisor} that $\ord_{\Xi\twisti}(g_j\twisti) = j-1$, and thus
\[\left (M_2 C_{i-1}\twist + E_1 C_{i-1}\twist d[\theta]\twisti - C_{i} d[\eta]\twistk{i} + (N_1 + M_m)C_{i} d[\theta]\twistk{i} + M_1C_{i}\right ) = 0\]
identically, which proves that $\{C_i\}$ satisfies the recursion equation \eqref{recurrence} and proves the proposition.
\end{proof}

\begin{corollary}\label{C:Expcorallary}
For $z\in \C_\infty$ we have the formal expression
\[\Exp_\rho\on
\left (\begin{matrix}
z\\
0 \\
\vdots\\
 0
\end{matrix}\right )=
\left (\begin{matrix}
z\\
0 \\
\vdots\\
 0
\end{matrix}\right )+
\sum_{i=0}^\infty 
\frac{z^{q^i}}{g_1\twisti(ff\twist\dots f\twistk{i-1})^n}\cdot 
\left (\begin{matrix}
g_1\\
g_2 \\
\vdots\\
g_n
\end{matrix}\right )\Bigg|_{\Xi\twisti}.\]
\end{corollary}
\begin{proof}
This follows from Theorem \ref{T:ExpCoef} by evaluating \eqref{gammaexpression} at $\Xi\twisti$, noticing that $g_j\twisti(\Xi\twisti)$ vanishes for $j\geq 2$, then solving for $c_{\ell,1}$.
\end{proof}
\begin{remark}
Theorem \ref{T:ExpCoef} and Corollary \ref{C:Expcorallary} should be considered generalizations Proposition 2.2.5 of \cite{AndThak90} and of the remark that follow it.
\end{remark}

\section{Coefficients of the logarithm function}\label{S:LogCoef}
The coefficients for the logarithm function associated to a rank 1 sign-normalized Drinfeld module were first studied by Anderson (see \cite[Prop. 0.3.8]{Thakur93}) and are described in \eqref{Drinfeldlog}.  The coefficients for the logarithm associated to the $n$th tensor power of the Carlitz module were studied by Anderson and Thakur, who give formulas for the lower right entry of these matrix coefficients in \cite[\S 2.1]{AndThak90}.  Recently, Papanikolas has written down explicit formulas using hyperderivatives in \cite[4.3.1 and Prop. 4.3.6(a)]{PLogAlg}.  In this section we develop new techniques to write down explicit formulas for the coefficients of the logarithm function $\Log_\rho\on$ associated to the $n$th tensor power of rank 1 sign-normalized Drinfeld modules.  Our method was inspired by ideas of Sinha from \cite{Sinha97} (see in particular his ``main diagram" in section 4.2.3).  However, where Sinha uses homological constructions to prove the commutativity of his diagram, we take a more direct approach using Anderson generating functions for ours.

For $g\in N = \Gamma(U,\mathcal O_E(-nV\twist))$ with $\deg(g) = mn + b$ with $0\leq b\leq q-1$, define the map
\[\varepsilon: N \to \C_\infty^n,\]
by writing $g$ in the $\sigma$-basis for $N$ described in Proposition \ref{P:review}(a),
\begin{equation}\label{gcoefficients}
g  = \sum_{i=0}^{m} \sum_{j=1}^n b_{j,i}\twistk{-i} (ff\twistinv \dots f\twistk{1-i})^n h_{n-j+1}\twistk{-i},
\end{equation}
where we denote $\bb_i = (b_{1,i},b_{2,i},\dots,b_{n,i})^\top$, and set
\begin{equation}\label{varepsdef}
\varepsilon(g) = \bb_0 + \bb_1 + \dots + \bb_m.
\end{equation}

We define the following diagram of maps, where we recall the definition of $\cM$ from \S \ref{S:Review} and of $\Omega$ from \eqref{Omegadef}
\begin{equation}\label{maindiagram}
\begin{diagram}
\Omega & \rTo^{\tau - f^n} & N & \rTo^{\varepsilon} & \C_\infty^n \\
\dTo^{T} & & & \ruTo_{\Exp_\rho\on} \\
\cM^n & \rTo^{-\RES_\Xi} & \C_\infty^n
\end{diagram}
\end{equation}
and where the maps $T$ and $\RES_\Xi$ are defined in \eqref{Tmapdef} and \eqref{RESdef} respectively.  We remark that using the operator $\tau - f^n$ one quickly sees that $\Omega \subset \cM$.

One of the main goals of this section is to prove that the diagram commutes.  Before we prove this, however, observe that if $\bu\in \C_\infty^n$ is not a period of $\Exp_\rho\on$, then $G_\bu\on \in \cM^n$ is not in the image of $T$ in diagram \ref{maindiagram}.  We require a preliminary result which allows us to modify $G_\bu\on$ to be in the image of $T$.  For $\bu \in \C_\infty^n$, write the coordinates of $G_\bu\on$ from \eqref{Gudef} as
\[G_\bu\on(t,y) = (k_1(t,y),k_2(t,y),\dots,k_n(t,y))^\top,\]
and then define the vector
\[\bk = (k_1([n]V), k_2(V\twist +[n-1]V),k_3([2]V\twist + [n-2]V),\dots, k_n([n-1]V\twist + V))^\top.\]
Next we define the vector valued function
\begin{equation}\label{Judef}
J_\bu\on := (j_1(t,y),j_2(t,y),\dots,j_n(t,y))^\top := G_\bu\on - \bk,
\end{equation}
and note that $j_k$ vanishes at the point $[k-1]V\twist+[n-k+1]V$.    Also denote
\[\bw := (w_1(t,y),w_2(t,y),\dots,w_n(t,y))^\top :=  (G-E_1\tau)(J_\bu\on) \in  \TT(y)^n,\]
where $G-E_1\tau$ is the operator defined in \eqref{GE1operator}, and let $\bz := \Exp_\rho\on(\bu)$ and denote its coordinates $\bz := (z_1,z_2,\dots,z_n)^\top$.
\begin{proposition}\label{P:bw}
The vector $\bw$ is in $H[t,y]^n$ and equals
\[\bw = 
\left ( \begin{matrix}
z_1\cdot(t-t(V\twist + [n-1]V))\\
z_2\cdot(t-t([2]V\twist + [n-2]V))\\
\vdots\\
z_{n-1}\cdot(t-t([n-1]V\twist + [1]V))\\
z_n\cdot(t-t([n]V\twist))
\end{matrix}\right).\]
\end{proposition}
\begin{proof}
By Proposition \ref{P:OperatorReview}(a), Proposition \ref{P:EuReview}(d) and \eqref{expfunctionalequation} we write
\begin{align}\label{GEpolynomial}
\bw' :=& (w_1'(t,y),w_2'(t,y),\dots,w_n'(t,y))^\top :=  (G-E_1\tau)(G_\bu\on)\\
\nonumber=& M_\delta\inv \big [-c_1\rho\on_y(\bz) + \rho\on_{t^2}(\bz) + (t+c_2)\rho\on_t(\bz) \\
\nonumber&+ (t^2+c_2t + c_4)\bz - (M_\tau + M_m)(\rho\on_y(\bz) + (y+c_1t+c_3)\bz)\big ].
\end{align}
In particular, from the last line of the above equation we see that $\bw'$ is a vector of rational functions in the space $H(t,y)$.  Further, for each rational function $w_i'$, the highest degree term in the numerator is $z_k t^2$ and the highest degree term in the denominator is $t$ (coming from the matrix $M_\delta\inv$).  Thus each $w_i'$ is a rational function in $H(t,y)$ of degree 2 (recall the $\deg(t)=2$) with $\tsgn(w_i') = z_k$.  We also observe that
\[(G-E_1\tau)(\bk)\in H(t,y)\]
and that each coordinate has degree 1.  This implies that each $w_i$ is in $H(t,y)$ and has degree 2 with $\tsgn (w_i)=z_k$.  Writing out the action of $G-E_1\tau$ on the coordinates of $J_\bu\on$ we obtain equations for $1\leq m\leq n$
\begin{equation}\label{jequations}
j_m \frac{g_{m+1}}{g_m} - j_{m+1} = w_m.
\end{equation}
From \eqref{GE1operator}, Proposition \ref{P:review}(d) and \eqref{GEpolynomial} we see that the only points at which $w_k$ might have poles are the zeros of $\delta_k$, namely the points
\[[k-1]V\twist + [n - k+1]V \quad\text{and}\quad [-(k-1)]V\twist - [n - k+1]V.\]
We remark that this shows that the coordinates of $\bw$ are regular at $\Xi\twisti$ for $i\geq 0$, even though the coordinates of $J_u\on$ themselves have poles at $\Xi\twisti$.  Recall from Proposition \ref{P:EuReview}(c) that the only poles of $j_k$ occur at $\infty$ and $\Xi\twisti$ for $i\geq 0$ and from \eqref{Judef} that $j_k$ vanishes at $[k-1]V\twist+[n-k+1]V$, while from Proposition \ref{P:review}(c) we observe that $g_{k+1}/g_k$ is regular away from infinity except for a simple pole at $[k-1]V\twist+[n-k+1]V$.  Therefore, the equations in \eqref{jequations} show that each coordinate $w_k$ is regular at the points $[k-1]V\twist + [n - k+1]V$ and $[-(k-1)]V\twist - [n - k+1]V$.  Thus, the coordinates $w_k$, being rational functions of degree 2 in $H(t,y)$, which are regular away from $\infty$, are actually in $H[t,y]$.  Further, we see from $\eqref{jequations}$ that each function $w_k$ vanishes at the point $[k]V\twist + [n - k]V$.  Since we know that $\tsgn(w_i)=z_k$, and since we've identified one of the zeros of $w_i$, we find using the Riemann-Roch theorem that $w_k = z_k(t-t([k]V\twist + [n - k]V).$
\end{proof}

\begin{theorem}\label{T:diagramcommutes}
Diagram \eqref{maindiagram} commutes.  In other words, for $h\in \Omega$, if we let
\[(\tau-f^n)(h) = g\in N\]
and let $-\RES_\Xi(T(h)) = \bu$, then we have $\varepsilon(g) = \Exp_\rho\on(\bu).$
\end{theorem}
\begin{proof}
First observe that the case for $n=1$ is proved in Theorem 5.1 of \cite{GP16}.  For the rest of the proof, assume $n\geq 2$.  Write $\deg(g) = mn + b$ with $0\leq b\leq q-1$, and for $0\leq i\leq m$ let $\bu_i$ be any element in $\C_\infty$ such that
\begin{equation}\label{uidef}
\Exp_\rho\on(\bu_i) = \bb_i,
\end{equation}
where $\bb_i$ is defined for $g\in H$ in \eqref{gcoefficients}.  The main method for the proof of Theorem \ref{T:diagramcommutes} is to write $T(h)$ in terms of Anderson generating functions.  To do this we compare the result of $T(h)$ under the $G-E_1\tau$ operator with the result of $J_{\bu_i}\on$ under $G-E_1\tau$ for $0\leq i\leq m$.

By the definition \eqref{GE1operator} we see that for any $\gamma \in \C_\infty(t,y)$
\begin{equation}\label{GE1ofT}
(G-E_1\tau)(T(\gamma)) = (0,\dots,0,g_1\twist(f^n \gamma - \gamma\twist))^\top.
\end{equation}
Since $f^n h - h\twist =  g$, using the notation of \eqref{gcoefficients} we can write
\begin{equation}\label{GEh}
(G-E_1\tau)(T(h)) = (0,\dots,0,g_1\twist\sum_{i=0}^{m} \sum_{j=1}^n b_{j,n-i+1}\twistk{-i} (ff\twistinv \dots f\twistk{1-j})^n h_{i}\twistk{-j})^\top.
\end{equation}

Next, we analyze $(G-E_1\tau)(J_{\bu_i})$ for $0\leq i\leq m$.  For the equations in $\eqref{jequations}$, if we set $i=1$, then we can solve for $j_2$.  We then substitute that into the equation for $i=2$, then solve that for $j_3$, and so on to get equations for $2\leq m\leq n$
\begin{equation}\label{j1equation}
j_1 \frac{g_{m+1}}{g_1} - j_{m+1}= w_m + w_{m-1}\frac{g_{m+1}}{g_n} + w_{n-2}\frac{g_{m+1}}{g_{n-1}} + \dots + w_1\frac{g_{m+1}}{g_2},
\end{equation}
where we understand $j_{n+1} = j_1\twist$.  We note that the functions $j_k$ and $w_k$ depend implicitly on $\bu_i$.  Using these equations we find that
\begin{equation}\label{Ibudef}
J_{\bu_i}\on + \left (0,w_1,w_2 + w_1\frac{g_3}{g_2},\dots,w_{n-1} + w_{n-2}\frac{g_n}{g_2} + \dots + w_1\frac{g_n}{g_{n-1}}\right )^\top =T\left ( j_1/g_1\right ).
\end{equation}
In general we will call $I_{\bu_i}\on := T\left ( j_1/g_1\right )$, noting the implicit dependence on $\bu_i$. Then by \eqref{GE1ofT} and by \eqref{j1equation} with $m=n$ we find
\begin{equation}\label{GE1eq1}
(G-E_1\tau)(I_{\bu_i}\on) = \left (0,\dots,0,w_n + w_{n-1}\frac{g_1\twist f^n}{g_n} + w_{n-2}\frac{g_1\twist f^n}{g_{n-1}} + \dots + w_1\frac{g_1\twist f^n}{g_2}\right )^\top.
\end{equation}
Denote the entry in the $n$th coordinate of the last equation as
\[\ell_{\bu_i}:= w_n + w_{n-1}\frac{g_1\twist f^n}{g_n} + w_{n-2}\frac{g_1\twist f^n}{g_{n-1}} + \dots + w_1\frac{g_1\twist f^n}{g_2},\]
so that we can restate \eqref{j1equation} with $m=n$ as
\begin{equation}\label{ellbu}
\frac{ j_1g_1\twist f^n}{g_1}=  j_1\twist+ \ell_{\bu_i}.
\end{equation}
Observe then by Proposition \ref{P:review}(b) and by Proposition \ref{P:bw} for $1\leq k\leq n$ that
\[w_{n-k+1}\frac{g_1\twist f^n}{g_{n-k+2}} = b_{n-k+1,i}g_1\twist h_{k},\]
so \eqref{GE1eq1} becomes
\[(G-E_1\tau)(I_{\bu_i}\on) = \left (0,\dots,0,g_1\twist(b_{n,i}h_1 + b_{n-1,i}h_2 + \dots + b_{1,i}h_n)\right )^\top\]
For the vector $\bu_i$ from \eqref{uidef}, denote
\begin{equation}\label{hbidef}
h_{\bu_i} = b_{n,i}h_1 + b_{n-1,i}h_2 + \dots + b_{1,i}h_n,
\end{equation}
and notice that $\ell_{\bu_i} = g_1\twist h_{\bu_i}.$  Specializing the above discussion to $i=0$, we see that the $n$th coordinate of $(G-E_1\tau)(I_{\bu_0}\on)$ matches up with the first $n$ terms of the $n$th coordinate of $(G-E_1\tau)(T(h))$ from \eqref{GEh}.

In general for $i>0$ we find that
\[(f\twistinv f\twistk{-2} \dots f\twistk{-k})^n \text{diag}\left (\frac{g_1}{g_1\twistk{-k}}, \dots, \frac{g_n}{g_n\twistk{-k}}\right )(I_{\bu_i}\on)\twistk{-k}
=
 T\left (\left (\frac{(ff\twist \dots f\twistk{k-1})^n j_1}{g_1}\right )\twistk{-k}\right ),\]
and to ease notation, for $k\geq 1$ let us denote the matrix
\[R_k := (f\twistinv f\twistk{-2} \dots f\twistk{-k})^n \text{diag}\left (\frac{g_1}{g_1\twistk{-k}}, \dots, \frac{g_n}{g_n\twistk{-k}}\right ).\]
Then we use \eqref{ellbu} $k$ times and apply the fact that $T$ is linear to obtain
\begin{align}\label{Rkequation}
\begin{split}
R_k (I_{\bu_i}\on)\twistk{-k} &= T\left (\left (\frac{ (ff\twist \dots f\twistk{k})^n j_1}{ g_1}\right )\twistk{-k}\right )\\
&= I_{\bu_i}\on +T\left ( \frac{\ell_{\bu_i} \twistk{-1}}{g_1}\right ) + \dots + T\left (\frac{(f\twistk{2-k} \dots f\twistk{-1})^n\ell_{\bu_i}\twistk{1-k}}{g_1\twistk{2-k}}\right ) + T\left (\frac{(f\twist \dots f\twistk{-1})^n \ell_{\bu_i}\twistk{-k}}{g_1\twistk{1-k}}\right ).
\end{split}
\end{align}
Then, if we let the operator $(G-E_1\tau)$ act on $R_k (I_{\bu_i}\on)\twistk{-k}$, applying \eqref{GE1ofT} to the last line of \eqref{Rkequation} we obtain a telescoping sum, and find that
 \[(G-E_1\tau) (R_k (I_{\bu_i}\on)\twistk{-k}) = \left (0,\dots,0,g_1\twist (ff\twistinv \dots f\twistk{1-k})^n h_{\bu_i}\twistk{-k}\right )^\top,\]
for $h_{\bu_i}$ defined in \eqref{hbidef}.  Note again that the terms in the last coordinate of the above vector are exactly the $in+1$ through $(i+1)n$ terms of the last coordinate of \eqref{GEh}.

Also, note that each term in the last line in \eqref{Rkequation} is coordinate-wise regular at $\Xi$ except $I_{\bu_i}\on$, so
\[\RES_\Xi(R_k (I_{\bu_i}\on)\twistk{-k}) = \RES_\Xi(I_{\bu_i}\on).\]
Then, recalling that each function $w_k$ and each quotient $j_{k+m}/j_k$ for $1\leq k,m\leq n$ is regular at $\Xi$, using definitions \eqref{Judef} and \eqref{Ibudef} together with Proposition \ref{P:ResidueReview}(a) we see that
\begin{equation}\label{Iresidue}
\RES_\Xi(I_{\bu_i}\on) = \RES_\Xi(J_{\bu_i}\on)= \RES_\Xi(G_{\bu_i}\on) = -\bu_i.
\end{equation}

Next, define
\[\textbf{I} = I_{\bu_0}\on + R_1 I_{\bu_1}\on + \dots + R_m I_{\bu_m}\on,\]
and observe by the above discussion that
\[(G-E_1\tau)(T(h) - \textbf{I}) = 0.\]
Further, for $h'\in \Omega$, by Lemma \ref{L:GE1andOmega} $(G-E_1\tau)(T(h')) = 0$ if and only if $h'\in \Omega_0$.  Since $\textbf{I}$ is the sum of elements in the image of the map $T$, we see that $T(h) - \textbf{I}$ is itself in the image of the map $T$.  Thus there is some $h'\in \Omega_0$ such that $T(h') = T(h) - \textbf{I}$.  Then, Proposition \ref{P:EuReview}(a) together with Proposition \ref{P:ResidueReview}(b) implies that for some $b\in \F_q[t,y]$
\[T(h) - \textbf{I}= T(h') = bG_{\Pi_n}\on.\]
Finally, by \eqref{Iresidue}, we calculate that
\[ \bu = -\RES_\Xi(T(h)) = -\RES_\Xi(\textbf{I} + bG_{\Pi_n}\on) = \bu_0 + \dots + \bu_m + b\Pi_n,\]
and thus by \eqref{varepsdef} and \eqref{uidef} we obtain
\[\Exp_\rho\on(\bu) = \Exp_\rho\on( \bu_0 + \dots + \bu_m + b\Pi_n) = \bb_0 + \dots + \bb_m =  \varepsilon(g).\]
\end{proof}

Having proven that diagram \eqref{maindiagram} commutes, we now apply the maps from the diagram to write down formulas for the coefficients of $\Log_\rho\on$.  First, for $d_j\in \C_\infty$ define the function 
\begin{equation}\label{cdef}
c(t,y) = d_nh_1+\dots +d_1h_n \in N \subset \bA,
\end{equation}
where $h_j$ are from Proposition \ref{P:review}(a).  Then define the formal sum
\begin{equation}\label{Bfunction}
B(t,y;\bd) = -\sum_{i=0}^\infty \frac{c\twisti}{(ff\twist f\twistk{2} \dots f\twisti)^n}
\end{equation}
for the vector $\bd = (d_1,\dots d_n)^\top \in \C_\infty^n$.  We remark that $B(t,y;\bd)$ is similar to the function $L_\alpha(t)$ defined by Papanikolas in \cite[\S 6.1]{P08}.
\begin{lemma}\label{L:convergence}
There exists a constant $C_0>0$ such that for $|d_j|\leq C_0$, the function $B$ is a rigid analytic function in $\Gamma(\cU,\cO_E(n(\Xi)))$, the space of rigid analytic functions on $\cU$ with at most a pole of order $n$ at $\Xi$.
\end{lemma}
\begin{proof}
Using \eqref{fdef} together with the facts that $\deg(\theta) = \deg(\alpha) = 2$, $\deg(\eta) = 3$ and $\deg(m) = q$, for $k\geq 1$ we find that $f\twistk{k} \in \TT_\theta[y]$ and
\[\left \lVert f\twistk{k} \right \rVert_\theta = q^{q^{k+1}}.\]
This implies that
\begin{equation}\label{Bdegestimate}
\left \lVert \frac{c(t,y)\twisti}{f\twist \dots f\twisti} \right \rVert_\theta = \left \lVert c(t,y) \right \rVert_\theta^{q^i}\cdot q^{\left (-n (q^{i+2} - q^2)/(q-1)\right )}.
\end{equation}
Since each $h_i\in \bA$, we see that $\rVert h_i \rVert_\theta$ is finite, and thus we can choose $C_0>0$ small enough such that for all $d_j\in \C_\infty$ with $|d_j|\leq C_0$ the norm
\[\left \lVert \frac{c(t,y)\twisti}{f\twist \dots f\twisti} \right \rVert_\theta \to 0\]
as $i\to \infty$.  This guarantees that for such $d_j$, the function
\[\sum_{i=0}^\infty \frac{c\twisti}{(f\twist f\twistk{2} \dots f\twisti)^n} \in \TT_\theta[y].\]
To finish the proof, we simply note that
\[B = -\frac{1}{f^n}\cdot \sum_{i=0}^\infty \frac{c\twisti}{(f\twist f\twistk{2} \dots f\twisti)^n}.\]
\end{proof}

\begin{theorem}\label{T:LogCoef}
For $\bz\in \C_\infty^n$ inside the radius of convegence of $\Log_\rho\on$, if we write
\[
\Log_\rho\on(\bz) = \sum_{i=0}^\infty P_i \bz\twisti,
\]
for $n\geq 2$, then for $\lambda$ the invariant differential defined in \eqref{tableofsymbols}
\begin{equation}\label{Piresidue}
P_i = \left \langle \Res_\Xi\left (\frac{g_j h_{n-k+1}\twisti}{(ff\twist \dots f\twisti)^n}\lambda\right )\right  \rangle_{1\leq j,k\leq n}
\end{equation}
and $P_i \in \Mat_n(H)$ for $i\geq 0$.
\end{theorem}
\begin{remark}
As for Theorem \ref{T:ExpCoef}, we remark that the above theorem holds for $n=1$, but again for ease of exposition in the proof we restrict to the case of $n\geq 2$.
\end{remark}
\begin{proof}
One quickly observes from the definition of $B$, that $(\tau-f^n)(B) = c(t,y)$, and thus $B\in \Omega$.  Denote $\bu := -\RES_\Xi(T(B)),$ so that by Theorem \ref{T:diagramcommutes} combined with the definition of the map $\varepsilon$ in \eqref{varepsdef} and \eqref{cdef}
\[\Exp_\rho\on(\bu) = \varepsilon(c(t,y)) = (d_1,d_2,\dots,d_n)^\top.\]
We wish to switch our viewpoint to thinking about $-\RES_\Xi(T(B))$ as a vector-valued function with input $(d_1,\dots,d_n)^\top\, ,|d_i|<C_0$, where $C_0$ is the constant defined in Lemma \ref{L:convergence}.  For $D_{0}$ the hyper-disk in $\C_\infty^n$ of radius $C_0$, we define $\tilde{B}:D_{0}\to \C_\infty^n,$ for $\bd\in D_0$, as
\[\tilde B(\bd) = -\RES_\Xi(T(B(t,y;\bd)).\]
From the above discussion, we find that
\[\Exp_\rho\on \circ\, \tilde B: D_0\to \C_\infty^n\]
is the identity function.  Writing out the definition for $\tilde B$ gives
\begin{equation}
\tilde B = 
-\left (\begin{matrix}
\Res_\Xi(Bg_1\lambda)\\
\vdots\\
\Res_\Xi(Bg_n\lambda)
\end{matrix}\right )
=
\left (\begin{matrix}
\Res_\Xi(\sum_{i=0}^\infty \sum_{j=1}^n \frac{(d_jh_{n-j+1})\twisti}{(ff\twist f\twistk{2} \dots f\twisti)^n}g_1\lambda)\\
\vdots\\
\Res_\Xi(\sum_{i=0}^\infty \sum_{j=1}^n \frac{(d_jh_{n-j+1})\twisti}{(ff\twist f\twistk{2} \dots f\twisti)^n}g_n\lambda)
\end{matrix}\right ),
\end{equation}
which we can express as an $\F_q$-linear power series with matrix coefficients
\[
\tilde B = \sum_{i=0}^\infty \left \langle \Res_\Xi\left (\frac{g_j h_{n-k+1}\twisti}{(ff\twist \dots f\twisti)^n}\lambda\right )\right  \rangle_{1\leq j,k\leq n}
\left (\begin{matrix}
d_1\\
\vdots\\
d_n
\end{matrix}\right )\twisti.
\]
We conclude that $\Exp_\rho\on\circ \,\tilde B$ is an $\F_q$-linear power series which as a function on $D_0$ is the identity.  Recall that $\Log_\rho\on$ is the functional inverse of $\Exp_\rho\on$ on the disk with radius $r_L$.  Thus, on the disk with radius $\min(C_0,r_L)$ we have the functional identity
\[
\tilde B = \Log_\rho\on.
\]
Comparing the coefficients of the above expression, and recalling that $f$, $g_i$ and $h_i$ are defined over $H$ finishes the proof.
\end{proof}

\begin{corollary}\label{C:LogCoefbottomrow}
For the coefficients $P_i$ for $i\geq 0$ of the function $\Log_\rho\on$, the bottom row of $P_i$ can be written as
\begin{equation}\label{bottomrow}
\left \langle \frac{h_{n-k+1}\twisti}{h_1(f\twist \dots f\twisti)^n}\bigg|_\Xi \right \rangle_{1\leq k\leq n}.
\end{equation}
\end{corollary}
\begin{proof}
Recall from \eqref{gidivisor} and \eqref{hidivisor} that $\ord_\Xi(g_j)=\ord_\Xi(h_j) = j-1$ and from \eqref{fdivisor} that $\ord_\Xi(f) = 1$.  This implies that, for $i=0$, each coordinate of the bottom row of the matrix \eqref{Piresidue} is regular at $\Xi$ except the last coordinate, which equals
\[\Res\left (\frac{g_nh_1}{f^n}\lambda \right ) = h_1(\Xi)\cdot \Res_\Xi\left (\frac{g_n}{f^n}\lambda \right ) .\]
Using various facts from Proposition \ref{P:review}, and observing that $h_1$ is regular at $\Xi$ and that $t-\theta$ is a uniformizer at $\Xi$, a short calculation gives
\[\Res_\Xi\left (\frac{g_n}{f^n}\lambda\right ) = \Res_\Xi\left (\frac{\delta_n}{h_2}\lambda\right ) = \Res_\Xi\left (-\frac{\nu_n\circ[-1]}{h_1(t-\theta)}\lambda\right ) = -\frac{\nu_n(-\Xi)}{h_1(\Xi)} \cdot \frac{1}{2 \eta + c_1\theta + c_3},\]
where $[-1]:E\to E$ denotes negation on $E$.  Finally, one calculates from the definition of $\nu_n$ from Proposition \ref{P:review}(d) that $\nu_n(-\Xi) = -2 \eta - c_1\theta - c_3$, which implies that
\begin{equation}\label{gnoverfn}
\Res_\Xi\left (\frac{g_n}{f^n}\lambda\right ) =  \frac{1}{h_1(\Xi)}.
\end{equation}
Thus, for $i=0$, the bottom row of \eqref{Piresidue} equals $(0,\dots,0,1)$, which is the bottom row of $Q_0=I$.

Then, for $i\geq 1$ note that the only functions in the bottom row of \eqref{Piresidue} which have zeros or poles at $\Xi$ are $g_n$ and $f^n$, and that the quotient $g_n/f^n$ has a simple pole at $\Xi$, thus
\[\Res_\Xi\left (\frac{g_n h_{n-k+1}\twisti}{(ff\twist \dots f\twisti)^n}\lambda\right ) = \frac{ h_{n-k+1}\twisti}{(f\twist \dots f\twisti)^n}\bigg|_\Xi \Res_\Xi\left (\frac{g_n}{f^n}\lambda\right ),\]
which completes the proof using \eqref{gnoverfn}.
\end{proof}

\begin{remark}
Theorem \ref{T:LogCoef} and Corollary \ref{C:LogCoefbottomrow} should be compared with the middle and last equalities in \eqref{Drinfeldlog}, respectively.
\end{remark}

\section{Zeta values}\label{S:ZetaValues}
In \cite{AndThak90}, Anderson and Thakur analyze the lower right coordinate of the coefficient $P_i$ of the logarithm function for tensor powers of the Carlitz module to obtain formulas similar to the ones we have provided in \S \ref{S:LogCoef}.  They then define a polylogarithm function and use their formulas to relate this to zeta values,
\[\zeta(n) = \sum_{\substack{a\in \F_q[\theta] \\ \sgn(a) = 1}} \frac{1}{a^n},\]
for all $n\geq 1$.  In this section, we prove a similar theorem for tensor powers of Drinfeld $\bA$-modules, but at the present it is unclear how to generalize the special polynomials which Anderson and Thakur used in their proof (the now eponymous Anderson-Thakur polynomials) to tensor powers of $\bA$-modules, and so we developed new techniques.  Presently, we only consider values of $n \leq q-1$ because these allow us to appeal to formulas from \cite{GP16}.

\begin{remark}\label{R:largezetavalues}
We remark that Pellarin, Angles, Ribeiro and Perkins develope a multivariable version of $L$-series in \cite{AnglesPellarin14}-\cite{AnglesPellarinRibeiro16}, \cite{PellarinPerkins16}, \cite{Perkins14a} and that it is possible that such considerations could enable one to obtain formulas for all zeta values; this is an area of ongoing study.
\end{remark}

To define a zeta function for a rank 1 sign-normalized Drinfeld module $\rho : \bA \to H[\tau]$, we first define the left ideal of $H[\tau]$ for an ideal $\fa \subseteq A$ by
\[
  J_{\fa} = \langle \rho_{\oa} \mid a \in \fa \rangle \subseteq H[\tau],
\]
where we recall that $\oa = \chi(a)$ from \S \ref{S:Background}.  Since $H[\tau]$ is a left principal ideal domain~\cite[Cor.~1.6.3]{Goss}, there is a unique monic generator $\rho_\fa\in J_\fa$, and we define $\partial(\rho_\fa)$ to be the constant term of $\rho_\fa$ with respect to $\tau$.  Let $\phi_{\fa} \in \Gal(H/K)$ denote the Artin automorphism associated to $\fa$, and let the $B$ be the integral closure of $A$ in $H$.  We define the zeta function associated to $\rho$ twisted by the parameter $b\in B$ to be
\begin{equation}\label{zetafunc}
\zeta_\rho(b;s) := \sum_{\fa \subseteq A} \frac{b^{\phi_\fa}}{\pd(\rho_\fa)^s},
\end{equation}

\begin{theorem}\label{T:ZetaValues}
For $b\in B$ and for $n\leq q-1$, there exists a vector $(*,\dots,*,C\zeta_\rho(b;n))^\top \in \C_\infty^n$ such that
\[\bd := \Exp_\rho\on\left (\begin{matrix}
*\\
\vdots\\
*\\
C\zeta_\rho(b;n)
\end{matrix}\right ) \in H^n,\]
where $C = \frac{(-1)^{n+1} h_1(-\Xi)}{\theta-t([n]V\twist)} \in H$.
\end{theorem}

\begin{remark}
We remark that the vector $\bd$ is explicitly computable as outlined in the proof of Theorem \ref{T:ZetaValues}.
\end{remark}

\begin{remark}
One would like to be able to express the above theorem in terms of evaluating $\Log_\rho\on$ at a special point and then getting a vector with $\zeta_\rho(n)$ as its bottom coordinate, as is done in \cite{AndThak90}.  However, one discovers that $\bd$ is not necessarily within the radius of convergence of $\Log_\rho\on$, and in fact $\bd$ can be quite large!  It is possible that one could use Thakur's idea from \cite[Thm. VI]{Thakur92} to decompose $\bd$ into small pieces which are each individually inside the radius of convergence of the logarithm for specific examples.
\end{remark}

Before giving the proof of Theorem \ref{T:ZetaValues} we require several additional definitions and preliminary results.  First, we denote $\bH$ as the Hilbert class field of $\bK$ (which is the fraction field of $\bA$), and denote $\Gal(\bH/\bK)$ as the Galois group of $\bH$ over $\bK$.  Then we observe that elements $\overline \phi \in \Gal(\bH/\bK)$ act on elements in the compositum field $H\bH$ by applying $\overline\phi$ to elements of $\bH$ and ignoring elements of $H$.  We also define the (isomorphic) Galois group $\Gal(H/K)$ and observe that elements $\phi \in \Gal(H/K)$ act on the compositum field $H\bH$ by applying $\phi$ to elements of $H$ and ignoring elements of $\bH$.  Let $\fp \subseteq A$ be a degree 1 prime ideal, to which there is an associated point $P = (t_0,y_0) \in E(\F_q)$ such that $\fp = (\theta-t_0, \eta-y_0)$, and let $\phi = \phi_{\fp} \in \Gal(H/K)$ denote the Artin automorphism associated to $\fp$ via class field theory.  Define the power sums
\begin{equation}
S_i(s)  = \sum_{a\in A_{i+}} \frac{1}{a^s}, \quad S_{\fp,i}(s) = \sum_{a\in \fp_{i+}} \frac{1}{a^s},
\end{equation}
where $A_{+}$ is the set of monic elements of $A$ and $A_{i+}$ is the set of monic, degree $i$ elements of $A$.  Then define the sums
\begin{equation}\label{zetadef}
\cZ_{(1)}(b;s) = b\sum_{i\geq 0} S_i(s) = b\sum_{A+} \frac{1}{a^s},\quad \cZ_\fp(b;s) = b^{\phi\inv}(-f(P)^{\phi\inv})^s\sum_{a \in \fp_{+}} \frac{1}{a^s}.
\end{equation}

We next prove a proposition which allows us to connect $\zeta_\rho(b;s)$ to the sums given above.  Much of our analysis follows similarly to that in \cite[\S 7-8]{GP16}, and we will appeal to it frequently throughout the remainder of the section.  
\begin{proposition}\label{P:Zfp}
Let $\fp_k$ for $2\leq k\leq h$ be the degree 1 prime ideals as described above which represent the non-trivial ideal classes of $A$ where $h$ is the class number of $A$ and set $\fp_1 = (1)$.  Then, for $s \in \Z$ we can write the zeta function
\[\zeta_\rho(b;s) = \cZ_{\fp_1}(b;s) + \dots + \cZ_{\fp_h}(b;s).\]
\end{proposition}
\begin{proof}
Define the sum
\[\tcZ_{\fp_k}(b;s) = \sum_{\fa \sim \fp_k} \frac{b^{\phi_\fa}}{\pd(\rho_{\fa})^s},\]
where the sum is over integral ideals $\fa$ equivalent to $\fp_k$ in the class group of $A$, and observe
\[\zeta_\rho = \sum_{k=1}^h \tcZ_{\fp_k}.\]
Then, for $1\leq k\leq h$, the fact that $\tcZ_{\fp_k}(b;s) = \cZ_{\fp_k}(b;s)$ follows from slight modifications to equations (98)-(100) and Lemma 7.10 from \cite{GP16}.
\end{proof}

Now, we let $\{w_i\}_{i=2}^\infty$ (the reader should not confuse these with the coordinates $w_i$ of $\bw$ from \S \ref{S:LogCoef}) be the sequence of linear functions with $\tsgn(w_i)=1$ and divisor
\begin{equation}\label{divgi}
\divisor(w_i) = (V^{(i-1)} - V) + (-V^{(i-1)}) + (V) - 3( \infty)
\end{equation}
and let $\{w_{\fp,i}\}_{i=2}^\infty$ be the sequence of functions with $\tsgn(w_{\fp,i}) = 1$ and divisor
\begin{equation}\label{divgfpi}
\divisor(w_{\fp,i}) = (V^{(i-2)} - V -P) + (-V^{(i-2)}) + (V) + (P) - 4(\infty).
\end{equation}
We now extend Theorem 6.5 from \cite{GP16} to values $1\leq s \leq q-1$, where we recall the definition of $\nu(t,y)$ from \eqref{nudiv}.

\begin{proposition}\label{P:Siextended}
For $1\leq s\leq q-1$ we find
\[S_i(s) =  \left (\frac{\nu^{(i)}}{w_i^{(1)}\cdot f^{(1)} \cdots f^{(i)}}\right )^s \Bigg|_\Xi, \quad S_{\fp,i}(s) = \left (\frac{\nu^{(i-1)}}{w_{\fp,i}^{(1)} \cdot f^{(1)} \cdots f^{(i-1)}}\right )^s \Bigg|_{\Xi}.\]
\end{proposition}
\begin{proof}
The proof of this proposition involves a minor alteration to the proof given for Proposition 6.5 in \cite{GP16}.  Namely, for the deformation $\cR_{i,s}(t,y)$ one sets $s=m$ (rather than $s=q-1$ as is done in \cite{GP16}) then one solves for $S_i(q-m)$ and sets $s=q-m$ to obtain the formula given above.  The proof for $S_{\fp,i}(s)$ is similar.
\end{proof}

Using equations (82) and (117) from \cite{GP16} we see that
\[\frac{\delta\twist}{w_i\twist}\bigg|_\Xi = \frac{f}{t-\theta}\bigg|_{V\twisti} = \frac{f(V\twisti)}{-\delta\twisti(\Xi)},\]
which inspires the definition 
\begin{equation} \label{Gdef}
  \cG := \frac{\beta + \overline{\beta} + c_1 \overline{\alpha}+ c_3}{\alpha - \overline{\alpha}} -
  \frac{\overline{\beta}^q + \overline{\beta} + c_1 \overline{\alpha} + c_3}{\overline{\alpha}^q - \overline{\alpha}},
\end{equation}
where we recall that $V = (\alpha,\beta)$ from \eqref{Vequation}, that $c_i\in \F_q$ are from \eqref{ecequation} and for $x\in H$ that $\overline x = \chi(x)$ as in \eqref{canoniso}.  Observe by \eqref{fdef} that $\cG\twisti(\Xi) = f(V\twisti)$ and hence
\begin{equation}\label{Giformula}
\frac{\delta\twist}{w_i\twist}\bigg|_\Xi = -\left (\frac{\cG}{\delta}\right )\twisti \bigg|_\Xi.
\end{equation}
Finally, we define
\[\tcG_b = \sum_{\overline \phi \in \Gal(\bH/\bK)} \ob^{\overline \phi} (\cG^{\overline \phi})^n.\]
\begin{proposition} \label{P:cG}
We have $f^n \tcG_b \in N$, where $N$ is the dual $\bA$-motive from \eqref{MNdef} and $f^n \tcG_b \in H[t,y]$.
\end{proposition}
\begin{proof}
Our function $\cG$ equals the function $\cF$ from \cite[(125)]{GP16} (there they set $\phi=\overline \alpha$ and $\psi = \overline \beta$), and so our function $\tcG_b$ differs from the function $g_b$ from \cite[(126)]{GP16} only by the $n$th power in our definition.    The proof of this theorem follows as in the proof of Theorem 8.7 from \cite{GP16}, replacing $\cF$ by $\cG^n$ and multiplying the divisors by a factor of $n$ where appropriate.  We arrive at the statement that the polar divisor of $\tcG_b$ equals $-n(\Xi) -(nq-\deg(b))(\infty)$, and that $\tcG_b$ vanishes with degree at least $n$ at $V$ so that $f^n\cdot \tcG \in N$ as desired.  Finally, since the coefficients of $f$ and $\cG$ are all in $H$, we conclude that $f^n \tcG_b \in H[t,y]$.
\end{proof}

We are now equipped to give the proof of Theorem \ref{T:ZetaValues}.

\begin{proof}[Proof of Theorem \ref{T:ZetaValues}]
Our starting point is Proposition \ref{P:Zfp},
\begin{equation}\label{cZfundamental}
\zeta_\rho(b;s) = \cZ_{\fp_1}(b;s) + \dots + \cZ_{\fp_h}(b;s)
\end{equation}
where we recall that for a degree 1 prime ideal $\fp$ and its associated Galois automorphism $\phi$
\begin{equation}\label{cZfp}
\cZ_\fp(b;n) = b^{\phi\inv}(-f(P)^{\phi\inv})^n\sum_{a \in \fp_{+}} \frac{1}{a^n} = b^{\phi\inv}(-f(P)^{\phi\inv})^n\sum_{i=0}^\infty S_{\fp,i}(n).
\end{equation}
If we let $[-1]$ denote the negation isogeny on $E$, by comparing divisors and leading terms of the functions in \eqref{nudiv} and \eqref{hidivisor} we find 
\begin{equation}\label{deltatwistn}
(\delta\twist)^n = \frac{(-1)^{n+1} (h_1)(h_1\circ [-1])}{t-t([n]V\twist)}.
\end{equation}
We will denote $C = \frac{(-1)^{n+1}(h_1\circ [-1])}{t-t([n]V\twist)}\bigg|_\Xi \in H$.  Combining \eqref{zetadef}, Proposition \ref{P:Siextended}, \eqref{Giformula} and \eqref{deltatwistn} we find
\begin{equation}\label{bZn}
\cZ_{(1)}(b;n) = \sum_{i=0}^\infty \frac{\ob\left (\left ( -f \cG\right )\twisti \right )^n}{C\cdot h_1\left (f^{(1)} \cdots f^{(i)}\right )^n} \Bigg|_\Xi.
\end{equation}

Next, we temporarily fix a prime $\fp = \fp_k$ for $2\leq k\leq h$.  The combination of equations (86) and (118) and Lemma 7.12 from \cite{GP16} gives
\begin{equation}\label{ellpformula}
\frac{1}{w_{\fp,i+1}\twist} = -\frac{f^{\phi\inv}}{t-\theta}\bigg|_{V\twisti} \cdot \frac{1}{\delta\twist(\Xi)}\cdot \frac{1}{f(P)^{\phi\inv}}= f^{\phi\inv}\bigg|_{V\twisti} \cdot \frac{1}{\delta\twist(\Xi)\delta\twisti(\Xi)}\cdot \frac{1}{f(P)^{\phi\inv}},
\end{equation}
since $t-\theta(V\twisti) = -\delta\twisti(\Xi)$.  Then, \eqref{cZfp} and Proposition \ref{P:Siextended} together with \eqref{ellpformula} and the fact that $S_{\fp,0}=0$ gives
\begin{equation}
\cZ_\fp(b;n) = (-1)^nb^{\phi\inv}\sum_{i=0}^\infty \left (\frac{f^{(i)}}{\delta\twist f^{(1)} \cdots f^{(i)}}\right )^n\Bigg|_{\Xi}\cdot \left (f^{\phi\inv}\right )^n\bigg|_{V\twisti}
\end{equation}
We observe by \eqref{fdef} and \eqref{Gdef} that $f^{\phi\inv}\bigl(V^{(i)} \bigr) = \bigl( \cG^{\ophi\inv} \bigr)^{(i)}(\Xi)$ and so by \eqref{deltatwistn} this gives
\begin{equation}\label{cZb}
\cZ_\fp(b;n) = \sum_{i=0}^\infty\frac{\ob^{\ophi\inv}\left ((-f \cG^{\ophi\inv})^n\right )^{(i)}}{ Ch_1\left ( f^{(1)} \cdots f^{(i)}\right )^n}\Bigg|_{\Xi}.
\end{equation}
Therefore, returning to \eqref{cZfundamental} we see by \eqref{bZn} and \eqref{cZb} that
\begin{equation}\label{cZn3}
\zeta_\rho(b;n) = \sum_{i=0}^\infty\sum_{\ophi\in \Gal(\bH/\bK)} \frac{\ob^{\ophi}\left ((-f \cG^{\ophi})^n\right )^{(i)}}{Ch_1 \left (f^{(1)} \cdots f^{(i)}\right )^n}\Bigg|_{\Xi}= \sum_{i=0}^\infty\frac{\left ((-1)^nf^n \tcG_b\right )^{(i)}}{ Ch_1\left (f^{(1)} \cdots f^{(i)}\right )^n}\Bigg|_{\Xi}.
\end{equation}
From the proof of Proposition \ref{P:cG} we see that $\deg(f^n \tcG_b) = n(q+1) + \deg(b)$ and from \eqref{gihigher} that $\deg(\sigma^j (h_k)) = n(j+1)+k$.  Let us write $\deg(b) = en + b'$ where $0\leq b'\leq n-1$ so that $\deg(f^n \tcG_b) = n(q+e+1) + b'$.  Since $(-1)^n(f \tcG_b)^n \in N$ by Proposition \ref{P:cG}, we can express it in terms of the basis from Proposition \ref{P:review}(a) with coefficients $d_{k,j}\in \overline K$,
\begin{equation}\label{fcFcoefficients}
(-1)^nf^n \tcG_b = \sum_{j=0}^{q+e} \sum_{k=1}^n d_{k,j} \sigma^j(h_{n-k+1}) = \sum_{j=0}^{q+e} \sum_{k=1}^n d_{k,j} (ff\twistinv \dots f\twistk{1-j})^n h_{n-k+1}\twistk{-j},
\end{equation}
where we comment that $d_{k,q+e} = 0$ for $k>b'$.  Since $(-1)^nf^n \tcG_b \in H[t,y]$ by Proposition \ref{P:cG}, a short calculation involving evaluating \eqref{fcFcoefficients} at $\Xi\twistk{k}$ for $0\leq k\leq q+e$ shows that $d_{k,j}\twistk{j} \in H$.  Substituting formula \eqref{fcFcoefficients} into \eqref{cZn3} and recalling that $f(\Xi) = 0$ gives
\[
\zeta_\rho(b;n) = \sum_{i=0}^\infty \frac{\sum_{j=0}^{\min(i,q+e)} \sum_{k=1}^n d_{k,j}\twisti h_{n-k+1}\twistk{i-j}}{C\cdot h_1\left (f^{(1)} \cdots f^{(i-j)}\right )^n} \Bigg|_\Xi.
\]
We observe that the terms of the above sum are the bottom row of the coefficients $P_i$ for $i\geq 0$ of $\Log_\rho\on$ from Corollary \ref{C:LogCoefbottomrow} up to the factor of $d_{k,j}\twisti/C$.  Then, since $\Log_\rho\on$ is the inverse power series of $\Exp_\rho\on$, if we label $\bd_{j} = (d_{1,j},\dots,d_{n,j})^\top \in \overline{K}^n$ for $0\leq j\leq q+e$ and sum over $i\geq 0$, then we find that there exists some vector $(*,\dots,*,C\zeta_\rho(b;n))^\top)$ such that
\[\left (\bd_0 +\bd_1\twist + \dots + \bd_{q+e}\twistk{q+e}\right ) =
\Exp_\rho\on\left (\begin{matrix}
*\\
\vdots\\
*\\
C\zeta_\rho(b;n)
\end{matrix}\right ) \in H^n
.\]
\end{proof}

\section{Transcendence implications}\label{S:Transcendence}
In this section we examine some of the transcendence applications of Theorem \ref{T:ZetaValues}.  This is in line with Yu's results on transcendence in \cite{Yu91} for the Carlitz module, where he proves that the ratio $\zeta_\rho(n)/\tpi^n$ is transcendental if $q-1\nmid n$ and rational otherwise.  Yu's work builds on Anderson's and Thakur's theorem in \cite{AndThak90}, where they express Carlitz zeta values as the last coordinate of the logarithm of a special vector in $A^n$ similarly to how we have done in Theorem \ref{T:ZetaValues}.  In the last couple decades, there has been a surge of research answering transcendence questions about arithmetic quantities in function fields, notably \cite{ABP04}, \cite{Brown01}, \cite{CP11}-\cite{CY07}, \cite{P08} and \cite{Yu97}.

\begin{theorem}\label{T:Transcendence}
Let $\rho$ be a rank 1 sign-normalized Drinfeld $\bA$-module, let $\pi_\rho$ be a fundamental period of $\exp_\rho$ and define $\zeta_\rho(b;n)$ as in \eqref{zetafunc} for $b\in B$, the integral closure of $A$ in the Hilbert class field of $K$.  Then for $m\leq q-1$ 
\[\dim_{\overline{K}}\Span_{\overline{K}} \{\zeta_\rho(b;1),\dots,\zeta_\rho(b;m),1,\pi_\rho,\dots,\pi_\rho^{m-1} \} = 2m.\]
\end{theorem}

Our main strategy for proving Theorem \ref{T:Transcendence} is to appeal to techniques Yu develops in \cite{Yu97}, where he proves an analogue of W\"ustholz's analytic subgroup theorem for function fields.  Yu's theorem applies to Anderson $\F_q[t]$-modules (called $t$-modules), whereas here we deal with $\bA$-modules.  Thus, we switch our perspective slightly by forgetting the $y$-action of $\rho\on$ in order to view $\rho\on$ as an $\F_q[t]$-module with extra endomorphisms provided by the $y$-action.  We will denote this $\F_q[t]$-module by $\hat{\rho}\on$.  Under the construction given in \S \ref{S:Review}, the $\F_q[t]$-module $\hat{\rho}\on$ corresponds to the dual $t$-motive $N$ when viewed as a $\C_\infty[t,\sigma]$-module (we have forgotten the $y$-action on $N$), which we denote by $N'$.  Before giving the proof of Theorem \ref{T:Transcendence} we require a couple of lemmas which ensure that $\hat{\rho}\on$ satisfies the correct properties as a $t$-module to apply Yu's theorem.

\begin{lemma}\label{L:Simple}
The Anderson $\F_q[t]$-module $\hat{\rho}\on$ is simple.
\end{lemma}
\begin{proof}
We recall the explicit functor between $t$-modules and dual $t$-motives as given in \cite[\S 5.2]{HJ16}.  For a $t$-module $\phi'$ with underlying algebraic group $J\subset \C_\infty^n$, define the dual $t$-motive $N(\phi')$ (note that this is denoted as $\check{M}(\underline E)$ in \cite[\S 5.2]{HJ16}) as $\Hom_{\F_q}(\G_a,J)$, the $\C_\infty[t,\sigma]$-module of all $\mathbb{F}_q$-linear homomorphisms of algebraic groups over $\C_{\infty}$.  One defines the $\C_\infty[t,\sigma]$-module structure on $N(\phi')$ by having $\C_\infty$ act by pre-composition with scalar multiplication, $\sigma$ act as pre-composition with the $q$th-power Frobenius and $t$ acting by $t\cdot m = \phi_t' m$ for $m\in N(\phi')$.  Note that $N(\hat{\rho}\on) = \Hom_{\F_q}(\G_a,\G_a^n)$ is naturally isomorphic to $\C_\infty[\tau]^n$ where $\sigma$ acts for $\bp(\tau)\in \C_\infty[\tau]^n$ by $\sigma\cdot \bp(\tau) = \bp(\tau)\cdot \tau$ and $\C_\infty$ acts by scalar multiplication on the right.  To maintain clarity, when we mean $\C_\infty$ with the action described above we will denote it as a $\C_\infty'$.  Also note that $N(\hat{\rho}\on)$ is isomorphic to $N' = \Gamma (U,\cO_E(nV))$ as $\C_\infty[t,\sigma]$-modules.

Now suppose that $J\subset \C_\infty^n$ is a non-trivial algebraic subgroup of $\C_\infty^n$ invariant under $\hat{\rho}\on(\F_q[t])$ defined by non-zero $\F_q$-linear polynomials $p_j(x_1,\dots,x_n) \in \overline K[x_1,\dots,x_n]$ for $1\leq j\leq m$.  We may assume that one of the polynomials, which we will denote as $p(x_1,\dots,x_n)$ has a non-zero term in $x_1$.  Then note that we have the injection of $\C_\infty'[t,\sigma]$-modules given by inclusion
\[\Hom_{\F_q}(\G_a,J) \hookrightarrow \Hom_{\F_q}(\G_a,\G_a^n),\]
which allows us to view $\Hom_{\F_q}(\G_a,J)$ as a $\C_\infty'[t,\sigma]$-submodule of $\C_\infty'[\tau]^n$, where the $\sigma$-action is given by right multiplication by $\tau$ as descrived above.  Then observe that the map given by
\[p_*:\Hom_{\F_q}(\G_a,\G_a^n) \to \Hom_{\F_q}(\G_a,\G_a)\]
is a $\C_\infty'$-vector space map, that $ \Hom_{\F_q}(\G_a,\G_a) \isom \C_\infty'[\tau]$ and that $\Hom_{\F_q}(\G_a,J) \subset \ker(p_*)$.  By considering degrees in $\tau$, we see that the $\C_\infty'$-vector subspace $(\C_\infty'[\tau],0,\dots,0)\subset \C_\infty'[\tau]^n$ maps to an infinite dimensional $\C_\infty'$-vector subspace of $\C_\infty'[\tau]$ under $p_*$.  This implies that $\Hom_{\F_q}(\G_a,\G_a^n)/\Hom_{\F_q}(\G_a,J)$ also has infinite dimension over $\C_\infty$.

On the other hand, recall that $N' = \Gamma (U,\cO_E(-nV\twist))$ is isomorphic to $N(\hat{\rho}\on)$ as $\C_\infty[t,\sigma]$-modules and that $N'$ is an ideal of the ring $\C_\infty[t,y]$.  Given a $\C_\infty[t,\sigma]$-submodule $J'\subset N'$ we may choose a non-zero element $h\in J'$, and we claim that $\sigma(h)$ is linearly independent from $h$ over $\F_q[t]$.  If not, then we would have
\begin{equation}\label{betaeq}
\beta h = f^n h\twistinv
\end{equation}
for some $\beta \in \F_q(t)$.  However, this implies that the rational function $f^n h\twistinv/h$ is fixed under the negation isogeny $[-1]$ on $E$, and in particular, for $i\neq 0$ we have
\begin{equation}\label{hxidiv}
\ord_{\Xi\twistk{i+1}}(h) - \ord_{-\Xi\twistk{i+1}}(h) +\ord_{-\Xi\twistk{i}}(h) -\ord_{\Xi\twistk{i}}(h) = 0.
\end{equation}
Since $h$ is a polynomial in $t$ and $y$, we see that $\ord_{\Xi\twistk{i}}(h) -\ord_{-\Xi\twistk{i}}(h) = 0$ for $|i|\gg 0$, thus \eqref{hxidiv} shows that $\ord_{\Xi\twistk{i}}(h) -\ord_{-\Xi\twistk{i}}(h) = 0$ for all $i$.  But from \eqref{betaeq} we see that
\[\ord_\Xi(f^n) + \ord_{\Xi\twistk{1}}(h) - \ord_{-\Xi\twistk{1}}(h) +\ord_{-\Xi}(h) -\ord_{\Xi}(h) = 0,\]
which is a contradiction, since $\ord_\Xi(f^n) = n$.
So $J'$ contains a rank 2 $\C_\infty[t]$-submodule and thus $J'$ has finite index in $N'$ as a $\C_\infty$-vector space.  We conclude that all the $\C_\infty[t,\sigma]$-submodules of $N'$ have finite index over $\C_\infty$ which contradicts our observation in the preceding paragraph, thus $\hat{\rho}\on$ must be simple as a $t$-module.
\end{proof}

\begin{lemma}\label{L:Endomorphisms}
The Anderson $\F_q[t]$-module $\hat{\rho}\on$ has endomorphism algebra equal to $\bA$.
\end{lemma}
\begin{proof}
Recall that endomorphisms of $\hat{\rho}\on$ are $\F_q$-linear endomorphisms $\alpha$ of $\C_\infty^n$ such that $\alpha\hat{\rho}\on_a = \hat{\rho}\on_a\alpha$ for all $a\in \F_q[t]$.  Thus $\bA$ is certainly contained in $\End(\hat{\rho}\on)$.  On the other hand, the $t$-module $\hat{\rho}\on$ and the $\bA$-module $\rho\on$ both have the same exponential function $\Exp_\rho\on$ and same period lattice $\Lambda_\rho\on$ (given in Proposition \ref{P:ResidueReview}(b)) associated to them.  We note, however, that whereas $\Lambda_\rho\on$ is a rank 1 $\bA$-module, when viewed as an $\F_q[t]$-module it is rank 2.  If we let $\End^0(\hat{\rho}\on) = \End(\hat{\rho}\on) \otimes_{\F_q[t]} \F_q(t)$ as an $\F_q(t)$-vector space, then \cite[Prop. 2.4.3]{BP02} implies that $[\End^0(\hat{\rho}\on):\F_q(t)]\leq 2$.  Since $\bA \subset \End(\hat{\rho}\on)$ is a rank 2 $\F_q[t]$-module, we see that $[\End^0(\hat{\rho}\on):\F_q(t)]= 2$, and thus $\End(\hat{\rho}\on)$ is a rank 2 $\F_q[t]$-module containing $\bA$.  Further, $\bA \otimes _{\F_q[t]} \F_q(t) = K$, and thus $\End^0(\hat{\rho}\on) = K$ as an $\F_q(t)$-vector space.  Since $\End(\hat{\rho}\on)$ is finitely generated over $\bA$, it is also integrally closed over $\bA$ and thus $\End(\hat{\rho}\on) = \bA$.

\end{proof}

\begin{proof}[Proof of Theorem \ref{T:Transcendence}]
This proof follows nearly identically to the proof of \cite[Prop. 4.1]{Yu97}.  First, assume by way of contradiction that
\[\dim_{\overline{K}} \Span_{\overline{K}}\{\zeta_\rho(b;1),\dots,\zeta_\rho(b;m),1,\pi_\rho,\dots,\pi_\rho^{m-1} \} < 2m,\]
so that there is a $\overline K$-linear relation among the $\zeta_\rho(b;i)$ and $\pi_\rho^j$ for $1\leq i\leq m$ and $0\leq j\leq m-1$.  Then, let $G_L$ be the 1-dimensional trivial $t$-module and set
\[G = G_L \times \left (\prod_{i=1}^m \hat{\rho}^{\otimes i}\right ) \times \left (\prod_{j=1}^{m-1} \hat{\rho}^{\otimes j}\right ).\]
For $1\leq i\leq m$ set $\bz_i = (*,\dots,*,C\zeta_\rho(b;i))^\top \in \C_\infty^i$ to be the vector from Theorem \ref{T:ZetaValues} such that $\Exp_\rho^{\otimes i}(\bz_i)\in H^i$, where $H$ is the Hilbert class field of $K$.  For $1\leq j\leq m-1$, let $\Pi_j \in \C_\infty^j$ be a fundamental period of $\Exp_\rho^{\otimes j}$ such that the bottom coordinate of $\Pi_j$ is an $H$ multiple of $\pi_\rho^j$ as described in Proposition \ref{P:ResidueReview}(c).  Define the vector
\[\bu = 1 \times \left (\prod_{i=1}^m \bz_i \right ) \times \left (\prod_{j=1}^{m-1} \Pi_j\right ) \in G(\C_\infty),\]
and note $\Exp_G(\bu) \in G(H)$, where $\Exp_G$ is the exponential function on $G$.  Our assumption that there is a $\overline K$-linear relation among the $\zeta_\rho(b;i)$ and $\pi_\rho^j$ implies that $\bu$ is contained in a $d\left [\F_q[t]\right ]$-invariant hyperplane of $G(\C_\infty)$ defined over $\overline K$.  This allow us to apply \cite[Thm. 3.3]{Yu97}, which says that $\bu$ lies in the tangent space to the origin of a proper $t$-submodule $H\subset G$.  Then, Lemmas \ref{L:Simple} and \ref{L:Endomorphisms} together with \cite[Thm 1.3]{Yu97} imply that there exists a linear relation of the form $a \zeta_\rho(b;j) + b \pi_\rho^j = 0$ for some $a,b \in H$ and $1\leq j\leq m-1$.  Since $\zeta_\rho(b;j)\in K_\infty$ and since $H\subset K_\infty$, this implies that $\pi_\rho^j \in K_\infty$.  However, we see from the product expansion for $\pi_\rho$ in \cite[Thm. 4.6]{GP16} that $\pi_\rho^j \in K_\infty$ if and only if $q-1|j$, which cannot happen because $j\leq m-1<q-1$.  This provides a contradiction, and proves the theorem.
\end{proof}

\begin{corollary}\label{C:Transcendence}
For $1\leq i\leq q-1$, the quantities $\zeta_\rho(b;i)$ are transcendental.  Further, the ratio $\zeta_\rho(b;i)/\pi_\rho^j \in \overline K$ for $0\leq j\leq q-1$ if and only if $i=j=q-1$.
\end{corollary}
\begin{proof}
The transcendence of $\zeta_\rho(b;i)$, as well as the statement that $\zeta_\rho(b;i)/\pi_\rho^j \notin \overline K$ for $i,j\neq q-1$ follows directly from Theorem \ref{T:Transcendence}.  On the other had, if $i=j=q-1$, then \cite[Thm. 2.10]{Goss80} guarantees that $\zeta_\rho(b;i)/\pi_\rho^j \in \overline K$.
\end{proof}

\begin{remark}
We comment that the statement in Corollary \ref{C:Transcendence} that $\zeta_\rho(b;i)$ for $i=1$ is transcendental can be recovered from Anderson's theorem on log-algebraicity from \cite[Thm. 5.1.1]{And94} together with Yu's analytic subspace theorem \cite{Yu97}.
\end{remark}

\section{Examples}\label{S:Examples}
\begin{example}\label{E:Carlitzexp}
In the case of tensor powers of the Carlitz module (see \cite{PLogAlg} for a detailed account on tensor powers of the Carlitz module), the formulas in Theorems \ref{T:ExpCoef} and \ref{T:LogCoef} for the coefficients of $\Exp_C\on$ and $\Log_C\on$ can be worked out completely explicitly using hyper-derivatives.  For instance, we find that $g_i = (t-\theta)^{i-1}$ and that the shtuka function is $f=(t-\theta)$, so the left hand side of \eqref{gammadef} is
\[\gamma_{\ell,i} = \frac{1}{(t-\theta)^{n-\ell}(t-\theta^q)^n\dots(t-\theta^{q^i})^n}.\]
We can expand $\gamma_{\ell,i}$ in terms of powers of $(t-\theta)$ by using hyper-derivatives, as described in \cite[\S 2.3]{PLogAlg}, namely
\[\gamma_{\ell,i} = \sum_{j=0}^\infty \partial_t^j(\gamma_{\ell,i})\bigg|_{t=\theta}\cdot (t-\theta)^j.\]
Using this we recover the coefficients of $\Exp_C\on$ as given in formula (4.3.2) and Proposition 4.3.6(b) from \cite{PLogAlg}.  The formulas for coefficients of the logarithm given in (4.3.4) and Proposition 4.3.6(a) from \cite{PLogAlg} can be derived similarly using Theorem \ref{T:LogCoef}.
\end{example}

\begin{example}\label{E:easyexample}
Let $E:y^2 = t^3-t-1$ be defined over $\F_3$, and note that $A = \F_q[t,y]$ has class number 1.  Then from \cite{Thakur93} we find that
\[f = \frac{y-\eta - \eta(t-\theta)}{t - \theta - 1}.\]
The Drinfeld module $\rho$ associated to the coordinate ring of $E$ is detailed in Example 9.1 in \cite{GP16}.  Further, the 2-dimensional Anderson $\bA$-module $\rho^{\otimes 2}$ is discussed in Example 7.1 of \cite{GreenPeriods17}, where formulas are given for the functions $g_i$ and $h_i$ from Proposition \ref{P:review}(a).
%
We calculate that the function $\cG$ from \eqref{Gdef} is $\cG =(\eta + y)/(\theta -t) -y$
and that for $b=1$ we can express $(-1)^2f^2\tcG_b = (f\cG)^2$ in the form given in \eqref{fcFcoefficients} as
\[(f\cG)^2 = \frac{-\eta^3}{\eta^2+1}h_1 + h_2 + \frac{\eta^{5/3}}{\eta^{2/3} + 1}h_1\twistinv f^2 + h_2\twistinv f^2 + \frac{-\eta^{5/9} + \eta^{1/3}}{\eta^{2/9} + 1}h_1\twistk{-2}(ff\twistinv)^2 + h_1\twistk{-2}(ff\twistinv)^2.\]
This allows us to write the formulas in Theorem \ref{T:ZetaValues} as
\[
\left (\begin{matrix}
1\\ \frac{-\eta^3}{\eta^2 + 1}
\end{matrix}\right )
+
\left (\begin{matrix}
1\\ \frac{\eta^5}{\eta^2 + 1}
\end{matrix}\right )
+
\left (\begin{matrix}
1\\ \frac{-\eta^5 + \eta^3}{\eta^2 + 1}
\end{matrix}\right )
=
\left (\begin{matrix}
0\\ 0
\end{matrix}\right )
=
\Exp_\rho\on\left (\begin{matrix}
*\\
-\frac{\eta^{3}}{\eta^{2} + 1}\zeta(2)
\end{matrix}\right ).
\]
Thus the special vector $\bz = (*,-\eta^{3}/(\eta^{2} + 1)\zeta(2))^\top$ is in the period lattice for $\Exp_\rho\on$ which by \cite[Thm. 6.7]{GreenPeriods17} implies that the bottom coordinate of $\bz$ is a $K$-multiple of $\pi_\rho^2$, the fundamental period associated to $\rho$.  Hence $\zeta(2)/\pi_\rho^2 \in K$ as implied by Goss's \cite[Thm. 2.10]{Goss80}.
\end{example}

\begin{example}
Now let $q=4$ and let $E/\F_q$ be defined by $y^2 + y = t^3 + c$,  where $c\in \F_4$ is a root of the polynomial $c^2 + c + 1 = 0$.  Then we know from \cite[\S 2.3]{Thakur93} that $A=\F_q[\theta,\eta]$ has class number 1, that $V = (\theta,\eta+1)$ and that
\[f = \frac{y + \eta + \theta^4(t + \theta)}{t + \theta}.\]
Setting the dimension $n=2$ and the parameter $b=1$, from \eqref{Gdef} we find that
\[\cG = \frac{\eta + y + 1}{\theta + t} + \frac{y^4 + y + 1}{t^4 + t}\]
and that $\tcG_1 = \cG^2$.  Then we compute the expansion from \eqref{fcFcoefficients} as
\begin{align*}
f^2\tcG_1 &= (\theta^4+\theta)^{-1}h_1 + h_2 + (\theta^4+\theta)^{1/4}h_1\twistinv f^2 + (\theta^4+\theta)^{1/2}h_2\twistinv f^2 + (\theta^4+\theta)^{3/16}h_1\twistk{-2}(ff\twistinv)^2\\
&+ (\theta^4+\theta)^{1/4}h_2\twistk{-2}(ff\twistinv)^2+ (\theta^4+\theta)^{-1/64}h_1\twistk{-3}(ff\twistinv f\twistk{-2})^2+ h_2\twistk{-3}(ff\twistinv f\twistk{-2})^2,
\end{align*}
whereupon Theorem \ref{T:ZetaValues} gives
\begin{align*}
\left (\begin{matrix}
1\\ (\theta^4+\theta)^{-1}
\end{matrix}\right )
+
\left (\begin{matrix}
(\theta^4+\theta)^{2}\\ (\theta^4+\theta)
\end{matrix}\right )
+
\left (\begin{matrix}
(\theta^4+\theta)^{4}\\ (\theta^4+\theta)^{3}
\end{matrix}\right )
+
\left (\begin{matrix}
1\\ (\theta^4+\theta)^{-1}
\end{matrix}\right )
&=
\left (\begin{matrix}
(\theta^4+\theta)^{2} + (\theta^4+\theta)^{4}\\ (\theta^4+\theta) + (\theta^4+\theta)^{3}
\end{matrix}\right )\\
&=
\Exp_\rho\on\left (\begin{matrix}
*\\
(\theta^4 + \theta)\inv\zeta(2)
\end{matrix}\right ).
\end{align*}
\end{example}

\end{document}